\documentclass[a4paper,11pt,reqno]{amsart}

\usepackage[latin1]{inputenc}

\usepackage{
amsfonts,
amsmath,
amssymb,
amsopn,
amsthm,
mathtools,
mathrsfs,
dsfont,
soul,
subfig,
graphicx,
enumitem,
xcolor,
caption,
cite
}
\mathtoolsset{showonlyrefs}

\usepackage{tikz}
\usetikzlibrary{decorations.pathreplacing}
\usetikzlibrary{shapes.geometric,calc,decorations.markings,math}
\tikzstyle{nodo}=[circle,draw,fill,inner sep=0pt,minimum size=%
1.5mm]
\tikzstyle{infinito}=[circle,inner sep=0pt,minimum size=0mm]

\usepackage[colorlinks=true,urlcolor=blue, citecolor=red,linkcolor=blue,
linktocpage,pdfpagelabels,bookmarksnumbered,bookmarksopen]{hyperref}

\usepackage{geometry}
\mathtoolsset{showonlyrefs}

\newtheorem{theorem}{Theorem}[section]
\newtheorem{lemma}[theorem]{Lemma}
\newtheorem{proposition}[theorem]{Proposition}
\newtheorem{corollary}[theorem]{Corollary}

\theoremstyle{remark}
\newtheorem{remark}[theorem]{Remark}

\theoremstyle{definition}

\newcommand\R{{\mathbb R}}

\newcommand\Nat{{\mathbb N}}
\newcommand\dx{{\,dx}}

\newcommand\NN{{\mathcal N}}
\newcommand\JJ{{\mathcal J}}

\newcommand\ee{{\mathcal E}}
\newcommand\eps{\varepsilon}

\newcommand\mud{{\mu_d}}

\newcommand\Omegan{{\Omega_n}}

\title[Non-uniqueness of normalized NLS ground states]{Non-uniqueness of normalized NLS ground states on polygons with homogeneous Neumann boundary conditions}

\author[S. Dovetta]{Simone Dovetta}
\address{S. Dovetta: Politecnico di Torino, Dipartimento di Scienze Matematiche ``G.L. Lagrange'', Corso Duca degli Abruzzi, 24, 10129, Torino, Italy}
\email{simone.dovetta@polito.it}

\author[E. Serra]{Enrico Serra}
\address{E. Serra: Politecnico di Torino, Dipartimento di Scienze Matematiche ``G.L. Lagrange'', Corso Duca degli Abruzzi, 24, 10129, Torino, Italy}
\email{enrico.serra@polito.it}

\author[L. Tentarelli]{Lorenzo  Tentarelli}
\address{L. Tentarelli: Politecnico di Torino, Dipartimento di Scienze Matematiche ``G.L. Lagrange'', Corso Duca degli Abruzzi, 24, 10129, Torino, Italy}
\email{lorenzo.tentarelli@polito.it}

\date{\today}



\begin{document}

\begin{abstract}
We provide a non-uniqueness result for normalized ground states of nonlinear Schr\"odinger equations with pure power nonlinearity on polygons with homogeneous Neumann boundary conditions, defined as global minimizers of the associated energy functional among functions with prescribed mass. Precisely, for nonlinearity powers slightly smaller than the $L^2$-critical exponent,  we prove that there always exists at least one value of the mass for which normalized ground states are not unique. 
\end{abstract}

\maketitle


\section{Introduction}

Given a polygon $\Omega\subseteq\R^2$, in this paper we consider normalized solutions of nonlinear Schr\"odinger equations in $\Omega$ with homogeneous Neumann conditions at the boundary $\partial\Omega$, i.e. solutions of the problem
\begin{equation}
	\label{eq:normNLS}
	\begin{cases}
		-\Delta u+\lambda u=|u|^{p-2}u & \text{in }\Omega\\
	\frac{\partial u}{\partial \nu}=0 & \text{on }\partial\Omega\\
	\|u\|_{L^2(\Omega)}^2=\mu\,. & 
\end{cases}
\end{equation}
Here, $p>2$, $\lambda\in\R$, $\mu>0$ are parameters of the problem, and $\nu$ is the unit outward normal vector at the boundary $\partial\Omega$ (except at the vertices). Solutions of \eqref{eq:normNLS} are called {\em normalized} because of the constraint on the $L^2$ norm of $u$, whose square is usually referred to as the {\em mass} of the solution.

Normalized solutions of nonlinear Schr\"odinger equations arise in a wide variety of applications, ranging e.g. from the mathematical modeling of Bose-Einstein condensates to Mean Field Games. Even though the first paper considering such solutions in $\R^d$ is perhaps \cite{CL82}, it is after \cite{J97} that the topic started collecting a constantly increasing interest. By now, the literature on the subject has grown enormously and the problem has been studied in many different settings (see e.g. \cite{BdV13, BJ, BJS, BS17, BS19, GJ, JL-22, S1, S2} for problems posed on the whole space $\R^d$, \cite{DDGS, DST-23, NTV-14, NTV15, NTV19, PV, PVY, SZ1, SZ2} for bounded domains, and \cite{AST-cmp, AST-cvpde, AST-16, ACT, BCJS, CJS, LLZ, NP, PS, PSV, ST, T} for metric graphs).

It is well-known that solutions of \eqref{eq:normNLS} are critical points of the {\em energy} functional $E_p(\cdot,\Omega):H^1(\Omega)\to\R$,
\begin{equation}
\label{eq-NLSenergy}
E_p(u,\Omega):=\frac12\|\nabla u\|_{L^2(\Omega)}^2-\frac1p\|u\|_{L^p(\Omega)}^p\,,
\end{equation}
constrained to the space of functions with fixed mass
\begin{equation}
\label{eq-vincolo}
H_\mu^1(\Omega):=\left\{u\in H^1(\Omega)\,:\,\|u\|_{L^2(\Omega)}^2=\mu\right\}.
\end{equation}
Note that, when a solution $u$ of \eqref{eq:normNLS} is found as a critical point of $E_p$ in $H_\mu^1(\Omega)$, the parameter $\lambda$ arises as a Lagrange multiplier associated to the mass constraint and depends on $u$ itself, as it is given by
\begin{equation}
	\label{eq-lambda}
	\lambda=\lambda_u:=\frac{\|u\|_{L^p(\Omega)}^p-\|\nabla u\|_{L^2(\Omega)}^2}{\|u\|_{L^2(\Omega)}^2}\,.
\end{equation}
Among all possible critical points of $E_p$, of particular interest are usually the so-called normalized, or energy, {\em ground states}, defined as global minimizers of $E_p$ in $H_\mu^1(\Omega)$, i.e. functions $u\in H_\mu^1(\Omega)$ such that
\begin{equation}
\label{eq-energygs}
E_p(u,\Omega)=\inf_{v\in H_\mu^1(\Omega)}E_p(v,\Omega)=:\ee_p(\mu,\Omega)\,.
\end{equation}
When $\Omega$ is a bounded subset of $\R^2$, it is straightforward to see that the Gagliardo-Nirenberg inequality
\begin{equation}
	\label{GN-bounded}
	\|u\|_{L^p(\Omega)}^p\lesssim_{\,p} \|u\|_{L^2(\Omega)}^{2}\,\|u\|_{H^1(\Omega)}^{p-2},\qquad\forall u\in H^1(\Omega)\,,
\end{equation}
guarantees the existence of energy ground states in $H_\mu^1(\Omega)$ for every $\mu>0$ and every $L^2$-subcritical power $p\in(2,4)$.
The $L^2$-critical exponent $p=4$ is known to be a sharp threshold on the nonlinearity for the behaviour of the energy ground state problem in two dimensions. Indeed, it can be easily proved that $\ee_p(\mu,\Omega)=-\infty$ for every $\mu>0$ and every $\Omega\subseteq\R^2$ when $p>4$, whereas the boundedness of $\ee_p$ and the existence of energy ground states when $p=4$ depends on the actual value of $\mu$ (see Proposition \ref{egs} below).

Since existence of energy ground states for $p<4$ is essentially settled, it is natural to wonder whether energy ground states at given mass are unique. Unsurprisingly, nothing is known about uniqueness of energy ground states with Neumann boundary conditions. In fact, the only uniqueness results we are aware of concern the full space $\R^d$ (see \cite{K-89}) and the ball with homogeneous Dirichlet boundary conditions (see \cite{NTV-14}). In both cases, however, a crucial role is played by the uniqueness of the positive solution of the NLS equation 
\begin{equation}
	\label{eq-NLS}
	-\Delta u+\lambda u=|u|^{p-2}u
\end{equation}
for fixed $\lambda$, a feature not available in a general bounded set. Besides, even knowing that \eqref{eq-NLS} admits a unique positive solution for every value of $\lambda$ would not be enough to ensure uniqueness of energy ground states since, as already pointed out, two critical points of $E_p$ with the same mass may well satisfy \eqref{eq:normNLS} with different values of $\lambda$.

The aim of the present paper is to provide the following negative answer to the question of uniqueness of energy ground states on polygons with homogeneous Neumann boundary conditions.
\begin{theorem}
\label{thm-main}
Let $\Omega\subset\R^2$ be a polygon. Then, there exists $\eps>0$ (depending on $\Omega$) such that, for every $p\in(4-\eps,4)$, there exists $\mu_p>0$ and $u_1,\,u_2\in H_{\mu_p}^1(\Omega)$ such that
\[
 u_1\neq u_2\qquad\text{and}\qquad E_p(u_1,\Omega)=E_p(u_2,\Omega)=\ee_p(\mu_p,\Omega),
\]
namely energy ground states of $E_p$ in $H_{\mu_p}^1(\Omega)$ are not unique.
\end{theorem}
The proof of Theorem \ref{thm-main} is based on a strategy developed in \cite{D-24} in the context of metric graphs. The idea is to show that, at least for one mass $\mu_p$, there are two energy ground states with mass $\mu_p$ that solve \eqref{eq-NLS} for two distinct values of $\lambda$, which then entails that they cannot coincide. We remark that, even though a priori this may seem the easiest way to obtain the non-uniqueness we seek, the occurrence of two energy ground states with equal mass but different $\lambda$ is actually quite surprising. Indeed, it has been shown in \cite[Theorem 2.5]{DST-20} (the proof therein is in the context of metric graphs, but it extends verbatim to our setting here) that, for all but at most countably many masses, all energy ground states with the same mass solve \eqref{eq:normNLS} with the same $\lambda$. Therefore, our proof of Theorem \ref{thm-main} reduces to showing that, at least for the values of $p$ covered by the theorem, this at most countable set of masses is not empty. This phenomenon is based essentially on two main elements: the fact, proved in \cite{DST-23}, that energy ground states are also ground states of the action functional on the associated Nehari manifold (see Section \ref{sec-actionGS} below for the detailed definitions), and the behaviour of energy and action ground states at the $L^2$-critical power $p=4$. In particular, with Proposition \ref{egs} below we prove that, when $p=4$, energy ground states on a polygon exist if and only $\mu\in(0,\overline\mu_{\overline\alpha}]$, where $\overline\alpha$ is the amplitude of the smallest internal angle of $\Omega$ and
\[
\overline\mu_{\overline\alpha}:=\frac{\overline\alpha}{2\pi}\overline\mu\,,
\]
with $\overline\mu>0$ a fixed constant independent of $\Omega$ (see \eqref{eq:muR2} below). The existence of $L^2$-critical energy ground states for a whole interval of masses {\em closed on the right} implies the existence of two $L^2$-critical action ground states such that the one with the smallest $\lambda$ has the largest mass. In other words, the mass of action ground states is not monotone in $\lambda$ when $p=4$. By continuity, this absence of monotonicity is preserved for slightly $L^2$-subcritical powers and this is in turn enough to obtain two energy ground states with equal mass and different values of $\lambda$.  

We highlight that in our proof of Theorem \ref{thm-main} Neumann boundary conditions are crucial to ensure that $L^2$-critical energy ground states exist {\em also} at the critical mass $\overline\mu_{\overline\alpha}$. This marks a sharp difference with respect to the analogous problem with homogeneous Dirichlet boundary conditions, for which it is well-known that $L^2$-critical energy ground states never exist at the critical mass (see e.g. \cite[Section 2]{DST-23}). In fact, since we already recalled that in the Dirichlet case energy ground states are known to be unique for every mass at least on the ball and on the whole $\R^2$, one may wonder whether the non-existence phenomenon unravelled in Theorem \ref{thm-main} is actually peculiar of Neumann boundary conditions and never occurs with Dirichlet ones. However, at present we are not able to extend the result of Theorem \ref{thm-main} to smooth domains. Indeed, when $\Omega$ is a smooth open bounded subset of $\R^2$, arguing as in the present paper it is easy to see that when $p=4$ energy ground states exist for every $\mu\in(0,\overline\mu_\pi)$, but we do not know whether this is true also when $\mu=\overline\mu_\pi$. 

Roughly speaking, the reason why the analysis at the critical mass is technically more delicate on smooth domains than on polygons is the following. In Section \ref{sec-energyGS} below, to prove that energy ground states with mass $\overline\mu_{\overline\alpha}$ do exist on a polygon whose smallest internal angle has amplitude $\overline\alpha$, we essentially show that any sequence of energy ground states with masses approaching $\overline\mu_{\overline\alpha}$ from below is precompact (and thus converge up to subsequences to a ground state with the critical mass). To do this, we argue by contradiction and assume that there exists a non-converging sequence of ground states with $\mu\to\overline\mu_{\overline\alpha}^-$. Adapting standard asymptotic analyses as the one developed e.g. in \cite{NT1} then shows that these ground states concentrate at the vertex of $\partial\Omega$ with angle $\overline\alpha$ and are exponentially small out of a suitable neighbourhood of that vertex. Up to 
$H^1$-negligible corrections, it is then possible to think of these ground states as $H^1$ functions supported on an infinite sector $\Sigma_{\overline\alpha}$ in $\R^2$ with amplitude $\overline\alpha$ (see \eqref{eq-Calpha} for the precise definition of $\Sigma_{\overline\alpha}$), and the available characterization of the ground states problem on sectors (see Section \ref{sec:prel} below) provides the desired contradiction. Note that, when passing from the polygon to the sector, no boundary correction term arises, since locally around a vertex the boundaries of the two sets coincide. Suppose now we try to adapt the same argument on a smooth domain. It is then easy to understand that the concentration of ground states at any point of the boundary would require a comparison with the corresponding ground states problem on a half-plane (i.e. an infinite sector $\Sigma_\pi$ with amplitude $\pi$). However, since the boundary of the set need not be flat, additional correction terms rooted in the discrepancy between $\partial\Omega$ and the boundary of the half-plane will appear. It turns out that, in general, these boundary correction terms are of higher order and sensibly affect the problems we are considering. A concrete instance of this fact can be seen e.g. comparing the asymptotic expansion we derive in Proposition \ref{asymp} below for the action ground state level at large frequencies on polygons with the analogous one obtained in \cite[Proposition 2.1]{NT2}: on polygons, the difference between this level and the corresponding one on $\Sigma_{\overline\alpha}$ is exponentially small as the frequency diverges, whereas on smooth domains the difference with the level on the half-space grows as the square root of the frequency. Similar estimates for the magnitude of the boundary correction terms on smooth domains have been obtained also in the context of normalized solutions for instance in \cite{PPVV}. Since the presence of these higher order correction terms does not allow us to recover the compactness we need, the existence of energy ground states with critical mass when $p=4$ on smooth domains of $\R^2$ (and more generally in $\R^n$ at the corresponding $L^2$-critical power), and thus the non-uniqueness of ground states at slightly $L^2$-subcritical powers as in Theorem \ref{thm-main}, remains at present an interesting open problem.

To conclude, we note that even though Theorem \ref{thm-main} is stated only for polygons, in view of the above discussion about smooth domains and of the fact that $\overline\mu_{\overline\alpha}<\overline\mu_\pi$ whenever $\overline\alpha\in(0,\pi)$, the argument developed in this paper generalizes to any open bounded subset $\Omega\subset\R^2$ with piecewise smooth boundary containing an angle with amplitude strictly smaller than $\pi$.


\subsection{Organization of the paper}
\label{subsec-organization}

The paper is organized as follows:
\begin{itemize}
	\item in Section \ref{sec:prel} we collect some preliminary results needed in the rest of the paper;
	\item in Section \ref{sec-actionGS} we deal with the auxiliary problem of the action ground states on the Nehari manifold; 
	\item in Section \ref{sec-energyGS} we discuss existence of energy ground states on polygons in the $L^2$-critical case;
	\item in Section \ref{sec-proof} we prove Theorem \ref{thm-main}.
\end{itemize}


\section{Preliminaries}
\label{sec:prel}

We start by recalling some basic notions we will largely use all along the paper and by establishing some preliminary results that will play an important role in the rest of our discussion.

In the present paper we exploit two different notions of ground states for nonlinear Schr\"odinger equations. As already stated in the Introduction, the first notion is that of {\em energy} ground states as in \eqref{eq-energygs}, defined as global minimizers of the energy $E_p$ given in \eqref{eq-NLSenergy} among all functions with fixed $L^2$ norm. The second notion is that of {\em action} ground states.
Given an open set $\Omega\subseteq\R^2$, a power $p>2$, and a real number $\lambda \in \R$, the NLS action functional at frequency $\lambda$ is the functional $J_p(\cdot,\lambda,\Omega) : H^1(\Omega) \to \R$ defined by
\begin{equation}
\label{eq-action}
J_p(u,\lambda,\Omega):=\frac12 \|\nabla u\|_{L^2(\Omega)}^2 + \frac\lambda{2} \|u\|_{L^2(\Omega)}^2- \frac1p \|u\|_{L^p(\Omega)}^p.
\end{equation}
The \emph{action ground states} at frequency $\lambda$ are the functions $u\in\NN_p(\lambda,\Omega)$ such that
\begin{equation}
\label{eq-actiongs}
J_p(u,\lambda,\Omega)=\JJ_p(\lambda,\Omega) := \inf_{v \in \NN_p(\lambda,\Omega)} J_p(v,\lambda,\Omega),
\end{equation}
where $\NN_p(\lambda,\Omega)$ is the Nehari manifold associated to $J_p$
\begin{equation}
\label{eq-nehari}
\NN_p(\lambda,\Omega): = \left\{u \in H^1(\Omega)\setminus\{0\}:\|\nabla u\|_{L^2(\Omega)}^2 + \lambda\|u\|_{L^2(\Omega)}^2 = \|u\|_{L^p(\Omega)}^p\right\}.
\end{equation}
It is well-known that
\[
J_p(v,\lambda,\Omega)=\frac{p-2}{2p}\|v\|_{L^p(\Omega)}^p,\qquad\forall v\in\NN_p(\lambda,\Omega),
\]
and that action ground states at frequency $\lambda$ are solutions of \eqref{eq-NLS} with homogeneous Neumann boundary conditions on $\partial\Omega$.

For a general set $\Omega$, the connection between energy and action ground states is provided by the following lemma, whose proof is omitted since it is exactly that of \cite[Theorem 1.3]{DST-23}.

\begin{lemma}
	\label{lem-energyaction}
	If $u$ is an energy ground state with mass $\mu$, then $u$ is an action ground state at frequency $\lambda:=\lambda_u$, with $\lambda_u$ defined by \eqref{eq-lambda}. Moreover, any other action ground state at frequency $\lambda$ is also an energy ground state with the same mass $\mu$.
\end{lemma}
For most sets, the result of Lemma \ref{lem-energyaction} cannot be improved. However, in a few special cases it is possible to provide a finer characterization of action and energy ground states. In $\R^2$, this is for instance the case of sectors, namely sets in the form
\begin{equation}
\label{eq-Calpha}
\Sigma_\alpha:=\left\{tz\in\R^2\,:\,t\in(0,+\infty)\,,\,z\in\omega_\alpha\right\}
\end{equation}
where, for any given $\alpha\in(0,2\pi)$, $\omega_\alpha$ is the open and connected subset of $S^1$ identified by the angles $\theta\in(0,\alpha)$. Note that $\Sigma_\pi$ is simply the upper half space above the $x$-axis. 

Since in the following ground states on sectors at the $L^2$-critical power $p=4$ will be important, we briefly comment on them here.

Fix then $p=4$ and let us first recall the well-known situation when $\Omega=\R^2$ (see e.g. \cite{BL-83,cazenave}). As for the action, ground states in $\NN_4(\lambda,\R^2)$ exist if and only if $\lambda>0$. For every $\lambda>0$ there exists a unique (up to translations) positive action ground state of $J_4(\cdot,\lambda,\R^2)$, the soliton $\phi_\lambda\in \NN_4(\lambda,\R^2)$. Moreover, $\phi_\lambda$ is a radially, strictly decreasing function (again up to translations) and 
\begin{equation}
\label{eq:phil}
\phi_\lambda(x)=\sqrt{\lambda}\phi_1(\sqrt{\lambda}\,x)\qquad\forall x\in\R^2\,,
\end{equation}
such that
\[
\|\phi_\lambda\|_{L^2(\R^2)}^2=\overline\mu\qquad\forall\lambda>0\,,
\] 
where the critical mass $\overline\mu$ is given by
\begin{equation}
\label{eq:muR2}
\overline\mu:=\frac2{K_2},
\end{equation}
with $K_2$ being the best constant in the Gagliardo-Nirenberg inequality
\[
\|u\|_{L^4(\R^2)}^4\leq K_2\|u\|_{L^2(\R^2)}^2\|\nabla u\|_{L^2(\R^2)}^2\qquad\forall u\in H^1(\R^2)\,,
\]
namely
\[
K_2:=\sup_{u\in H^1(\R^2)\setminus\{0\}}\frac{\|u\|_{L^4(\R^2)}^4}{\|u\|_{L^2(\R^2)}^2\|\nabla u\|_{L^2(\R^2)}^2}\,.
\]
Conversely, the ground state energy level $\ee_4(\mu,\R^2)$ satisfies
\begin{equation}
\label{eq:levER2}
\ee_4(\mu,\R^2)=\begin{cases}
0 & \text{if }\mu\leq\overline\mu\\
-\infty & \text{if }\mu>\overline\mu
\end{cases}
\end{equation}
and it is attained if and only if $\mu=\overline\mu$. Moreover, up to translations and change of sign the energy ground states with mass $\overline\mu$ are exactly the solitons $\phi_\lambda$, for every $\lambda>0$, so that
\[
E_4(\phi_\lambda,\R^2)=0\qquad\forall\lambda>0\,.
\] 
The next propositions fully describe the $L^2$-critical ground states problems on sectors in the plane.
\begin{proposition}
	\label{prop:AGSsector}
	Let $\alpha\in(0,2\pi)$ and $\Sigma_\alpha$ be the set in \eqref{eq-Calpha}. Then, for every $\lambda>0$,
	\begin{equation}
	\label{eq:levJSa}
	\JJ_4(\lambda,\Sigma_\alpha)=\begin{cases}
	\frac\alpha{2\pi}\JJ_4(\lambda,\R^2) & \text{if }\alpha\in(0,\pi]\\
	\frac{1}{2}\JJ_4(\lambda,\R^2) & \text{if }\alpha\in(\pi,2\pi)\,.
	\end{cases}
	\end{equation}
	Moreover, 
	\begin{itemize}
		\item[(i)] if $\alpha\in(0,\pi]$, then the unique positive action ground state of $J_4(\cdot,\lambda,\Sigma_\alpha)$ in $\NN_4(\lambda,\Sigma_\alpha)$ is the restriction of $\phi_\lambda$ to $\Sigma_\alpha$ (up to translations in the case $\alpha=\pi$);
		
		\item[(ii)] if $\alpha\in(\pi,2\pi)$, then $\JJ_4(\lambda,\Sigma_\alpha)$ is not attained.
	\end{itemize}
\end{proposition}

\begin{proposition}
	\label{prop:EGSsector}
	Let $\alpha\in(0,2\pi)$, $\Sigma_\alpha$ be the set in \eqref{eq-Calpha}, and
	\begin{equation}
	\label{eq:mua}
	\overline\mu_\alpha:=\begin{cases}
	\frac\alpha{2\pi}\overline\mu & \text{if } \alpha\in(0,\pi]\\
	\frac{1}2\overline\mu & \text{if } \alpha\in(\pi,2\pi)\,.
	\end{cases}
	\end{equation}
	Then
	\begin{equation}
	\label{eq:levESa}
	\ee_4(\mu,\Sigma_\alpha)=\begin{cases}
	0 & \text{if }\mu\leq\overline{\mu}_\alpha\\
	-\infty & \text{if }\mu>\overline{\mu}_\alpha\,.
	\end{cases}
	\end{equation}
	Moreover,
	\begin{itemize}
		\item[(i)] if $\alpha\in(0,\pi]$, then $\ee_4(\mu,\Sigma_\alpha)$ is attained if and only if $\mu=\overline\mu_\alpha$, and the positive energy ground states of $E_4(\cdot,\Sigma_\alpha)$ with mass $\overline{\mu}_\alpha$ are the restriction of $\phi_\lambda$ to $\Sigma_\alpha$, for every $\lambda>0$ (up to translations in the case $\alpha=\pi$);

		\item[(ii)]  if $\alpha\in(\pi,2\pi)$, then $\ee_4(\mu,\Sigma_\alpha)$ is not attained for any $\mu>0$.
	\end{itemize}
\end{proposition}
The content of Propositions \ref{prop:AGSsector}-\ref{prop:EGSsector} is perhaps well-known, and a part of it can be found e.g. in \cite{K18}. However, since we did not find a single reference to quote for all the above results, we provide a quick proof for the sake of completeness. 

We will exploit the following symmetrization procedures. For a measurable function $u:\Sigma_\alpha\to\R$, let $\mu:\R^+\to\R^+$ be its distribution function
\[
\mu(t):=|\left\{x\in\Sigma_\alpha\,:\,|u(x)|>t\right\}|
\]
and define $u^\sharp:\R^+\to\R^+$ as
\[
u^\sharp(s):=\inf\left\{t\geq0\,:\,\mu(t)<s\right\}\,.
\]
For $\alpha\in(0,\pi]$ we define the symmetric rearrangement of $u$ on $\Sigma_\alpha$ as the function $u^*:\Sigma_\alpha\to\R^+$ given by
\begin{equation}
\label{eq:u^*1}
u^*(x):=u^\sharp(\alpha|x|^2/2)\qquad\forall x\in\Sigma_\alpha\,.
\end{equation}
Observe that, by construction, $u^*$ is radially symmetric and nonincreasing on $\Sigma_\alpha$. Furthermore, if $u\in H^1(\Sigma_\alpha)$, then
\begin{equation}
\label{eq:equi1}
\|u\|_{L^r(\Sigma_\alpha)}=\|u^*\|_{L^r(\Sigma_\alpha)}\qquad\forall r\geq2
\end{equation}
and the following P\'olya-Sz\"ego inequality holds (see e.g. \cite[Theorem 3.7]{PT} or \cite[Section 2]{K18})
\begin{equation}
\label{eq:PS1}
\|\nabla u^*\|_{L^2(\Sigma_\alpha)}\leq\|\nabla u\|_{L^2(\Sigma_\alpha)}\,,
\end{equation}
and if equality occurs for some $u$ such that the set $|\left\{x\in\Sigma_\alpha\,:\,u^*(x)>0\,,\,\nabla u^*(x)=0\right\}|=0$, then (up to translations in the case $\alpha=\pi$) the function $u$ coincides with $u^*$ (a detailed proof of this last property can be found in \cite{BZ} for symmetric rearrangements in the whole space, but the argument is analogous on $\Sigma_\alpha$ for $\alpha\in(0,\pi]$).

On the other hand, for $\alpha\in(\pi,2\pi)$ we define the symmetric rearrangement of $u$ on $\Sigma_\pi$ as the function $u^*:\Sigma_\pi\to\R^+$ given by
\begin{equation}
\label{eq:u^*2}
u^*(x):=u^\sharp(\pi|x|^2/2)\qquad\forall x\in\Sigma_\pi\,.
\end{equation}
It is readily seen again that $u^*$ is radially symmetric and nonincreasing on $\Sigma_\pi$ and, if $u\in H^1(\Sigma_\alpha)$,
\begin{equation}
\label{eq:equi2}
\|u\|_{L^r(\Sigma_\alpha)}=\|u^*\|_{L^r(\Sigma_\pi)}\qquad\forall r\geq2\,.
\end{equation}
Furthermore, combining the standard argument for the proof of the P\'olya-Sz\"ego inequality with the characterization of the optimal isoperimetric sets in $\Sigma_\alpha$ given in \cite[Chapter 2]{bandle}, it can be proved that
\begin{equation}
\label{eq:PS2}
\|\nabla u^*\|_{L^2(\Sigma_\pi)}\leq\|\nabla u\|_{L^2(\Sigma_\alpha)}
\end{equation}
and that the inequality is strict if $u>0$ on $\Sigma_\alpha$.

We are now in position to prove the above results for ground states on sectors.
\begin{proof}[Proof of Proposition \ref{prop:AGSsector}]
	Consider first the case $\alpha\in(0,\pi]$. Observe that, for every $u\in \NN_4(\lambda,\Sigma_\alpha)$, if $u^*$ denotes its symmetric rearrangement on $\Sigma_\alpha$ as in \eqref{eq:u^*1}, then by \eqref{eq:equi1} and \eqref{eq:PS1}
	\[
	\sigma_\lambda(u^*):=\left(\frac{\|\nabla u^*\|_{L^2(\Sigma_\alpha)}^2+\lambda\|u^*\|_{L^2(\Sigma_\alpha)}^2}{\|u^*\|_{L^4(\Sigma_\alpha)}^4}\right)^{1/2}\leq \left(\frac{\|\nabla u\|_{L^2(\Sigma_\alpha)}^2+\lambda\|u\|_{L^2(\Sigma_\alpha)}^2}{\|u\|_{L^4(\Sigma_\alpha)}^4}\right)^{1/2}=1\,.
	\]
	Hence, since $\sigma_\lambda(u^*)u^*\in\NN_4(\lambda,\Sigma_\alpha)$ by construction, we obtain that
	\[
	\begin{split}
	\JJ_4^{rad}(\lambda,\Sigma_\alpha):=&\,\inf_{\substack{v\in\NN_4(\lambda,\Sigma_\alpha)\setminus\{0\} \\ v \text{ radial} }}J_4(v,\lambda,\Sigma_\alpha)\leq J_4(\sigma_\lambda(u^*)u^*,\lambda,\Sigma_\alpha)\\
	=&\,\frac14\sigma_\lambda(u^*)^4\|u^*\|_{L^4(\Sigma_\alpha)}^4\leq\frac14\|u\|_{L^4(\Sigma_\alpha)}^4=J_4(u,\lambda,\Sigma_\alpha)\,.
	\end{split}
	\]
	Taking the infimum over $u\in\NN_4(\lambda,\Sigma_\alpha)$ then yields $\JJ_4^{rad}(\lambda,\Sigma_\alpha)\leq\JJ_4(\lambda,\Sigma_\alpha)$, and since the reverse inequality is trivially true we obtain
	\begin{equation}
	\label{eq:Jr=J}
	\JJ_4^{rad}(\lambda,\Sigma_\alpha)=\JJ_4(\lambda,\Sigma_\alpha)\,.
	\end{equation}
	We now check that 
	\begin{equation}
	\label{eq:Jrad}
	\displaystyle\JJ_4^{rad}(\lambda,\Sigma_\alpha)=\frac\alpha{2\pi}\JJ_4(\lambda,\R^2)
	\end{equation}
	and that the unique (up to translations in the case $\alpha=\pi$) positive ground state of $\JJ_4^{rad}(\lambda,\Sigma_\alpha)$ is the restriction of the soliton $\phi_\lambda$ to $\Sigma_\alpha$. Indeed, the restriction of $\phi_\lambda$ to $\Sigma_\alpha$ is radial, belongs to $\NN_4(\lambda,\Sigma_\alpha)$, and
	\[
	\|\nabla\phi_\lambda\|_{L^2(\Sigma_\alpha)}^2=\frac{\alpha}{2\pi}\|\nabla\phi_\lambda\|_{L^2(\R^2)}^2\,,\,\quad\|\phi_\lambda\|_{L^r(\Sigma_\alpha)}^r=\frac\alpha{2\pi}\|\phi_\lambda\|_{L^r(\R^2)}^r\quad\forall r\geq2\,,
	\]
	directly implying that $\displaystyle\JJ_4^{rad}(\lambda,\Sigma_\alpha)\leq J_4(\phi_\lambda,\lambda,\Sigma_\alpha)=\frac\alpha{2\pi}\JJ_4(\lambda,\R^2)$. Conversely, for any radial function $u\in\NN_4(\lambda,\Sigma_\alpha)$,
	\[
	\|\nabla u\|_{L^2(\Sigma_\alpha)}^2=\alpha\int_{\R^+}|f'(\rho)|^2\rho\,d\rho\,,\quad\,\|u\|_{L^r(\Sigma_\alpha)}^r=\alpha\int_{\R^+}|f(\rho)|^r\rho\,d\rho\quad\forall r\geq2\,,
	\]
	with $f:\R^+\to\R$ being the radial profile of $u$. Then, letting $v$ be the radial function on $\R^2$ with the same radial profile $f$ of $u$ one has
	\[
	\|\nabla v\|_{L^2(\R^2)}^2=\frac{2\pi}\alpha\|\nabla u\|_{L^2(\Sigma_\alpha)}^2\,,\,\quad\|v\|_{L^r(\R^2)}^r=\frac{2\pi}{\alpha}\|u\|_{L^r(\Sigma_\alpha)}^r\quad\forall r\geq2\,,
	\]
	so that $v\in\NN_4(\lambda,\R^2)$ and $\displaystyle\JJ_4(\lambda,\R^2)\leq J_4(v,\lambda,\R^2)=\frac{2\pi}{\alpha}J_4(u,\lambda,\Sigma_\alpha)$. Taking the infimum over all radial functions $u\in\NN_4(\lambda,\Sigma_4)$ yields $\displaystyle\JJ_4^{rad}(\lambda,\Sigma_\alpha)\geq\frac\alpha{2\pi}\JJ_4(\lambda,\R^2)$, and thus \eqref{eq:Jrad}. The above computations also show that the only positive radial function in $\NN_4(\lambda,\Sigma_4)$ attaining $\JJ_4(\lambda,\Sigma_\alpha)$ is the restriction of $\phi_\lambda$ to $\Sigma_\alpha$. 
	
	Now, by \eqref{eq:Jr=J} and \eqref{eq:Jrad} we have the first line of \eqref{eq:levJSa}. Moreover, since we already know that the restriction of $\phi_\lambda$ to $\Sigma_\alpha$ is the unique positive radial action ground state in $\NN_4(\lambda,\Sigma_\alpha)$, to complete the proof of item (i) it is enough to exclude the existence of any other positive action ground state (up to translations in the case $\alpha=\pi$). However, if $u\in\NN_4(\lambda,\Omega)$ is such that $\JJ_4(\lambda,\Sigma_\alpha)=J_4(u,\lambda,\Sigma_\alpha)$, then by \eqref{eq:equi1} and \eqref{eq:PS1} its symmetric rearrangement $u^*$ on $\Sigma_\alpha$ is an action ground state too, and equality occurs in \eqref{eq:PS1}. Since $u^*$ is radial by construction, it is the restriction of $\phi_\lambda$ to $\Sigma_\alpha$. In particular, $|\left\{x\in\Sigma_\alpha\,:\,u^*(x)>0\,,\,\nabla u^*(x)=0\right\}|=0$, implying that $u$ coincides with $u^*$ (up to translations in the case $\alpha=\pi$). This concludes the proof of item (i).
	
	Take now $\alpha\in(\pi,2\pi)$. Relying on \eqref{eq:equi2} and \eqref{eq:PS2} and arguing as in the first part of the proof gives $\displaystyle\JJ_4(\lambda,\Sigma_\alpha)=\JJ_4^{rad}(\lambda,\Sigma_\pi)=\JJ_4(\lambda,\Sigma_\pi)$ (observe also that the restrictions to $\Sigma_\pi$ of the functions in $C_0^\infty(\R^2)$ are dense in $H^1(\Sigma_\pi)$ and a suitable translation of any such function can be thought of as defined on $\Sigma_\alpha$, immediately yielding $\JJ_{4}(\lambda,\Sigma_\alpha)\leq \JJ_4(\lambda,\Sigma_\pi)$). Together with the results already proved for $\Sigma_\pi$, this completes the proof of \eqref{eq:levJSa}. Finally, assume by contradiction that there exists an action ground state $u\in\NN_4(\lambda,\Sigma_\alpha)$, and with no loss of generality let $u\geq0$. Since $u$ solves \eqref{eq-NLS} on $\Sigma_\alpha$ with homogeneous Neumann boundary conditions, we obtain $u>0$ on $\Sigma_\alpha$. This ensures the validity of the strict inequality in \eqref{eq:PS2}, so that the symmetric rearrangement $u^*$ of $u$ on $\Sigma_\pi$ defined in \eqref{eq:u^*2} satisfies
	\[
	\sigma_\lambda(u^*):=\left(\frac{\|\nabla u^*\|_{L^2(\Sigma_\pi)}^2+\lambda\|u^*\|_{L^2(\Sigma_\pi)}^2}{\|u^*\|_{L^4(\Sigma_\pi)}^4}\right)^{1/2}<1\,,
	\]
	in turn yielding 
	\[
	\JJ_4(\lambda,\Sigma_\pi)\leq J_4(\sigma_\lambda(u^*)u^*,\lambda,\Sigma_\pi)<J_4(u,\lambda,\Sigma_\alpha)=\JJ_4(\lambda,\Sigma_\alpha)\,.
	\]
	Since this is impossible by \eqref{eq:levJSa}, we conclude.
\end{proof}
\begin{proof}[Proof of Proposition \ref{prop:EGSsector}]
	Let us start with $\alpha\in(0,\pi]$. By \eqref{eq:levER2}, for every $\mu>\overline\mu$ there exists $u\in H_\mu^1(\R^2)$ such that $E_4(u,\R^2)<0$. With no loss of generality we can take $u$ to be radial on $\R^2$. Setting $u_s(x):=su(sx)$ for every $x\in\R^2$, $s>0$, and restricting $u_s$ to $\Sigma_\alpha$ it follows that $u_s\in H_{\frac{\alpha\mu}{2\pi}}^1(\Sigma_\alpha)$ for every $s>0$ and
	\[
	\ee_4\left(\frac{\alpha\mu}{2\pi},\Sigma_\alpha\right)\leq\lim_{s\to\infty}E_4(u_s,\Sigma_\alpha)=\lim_{s\to\infty}s^2E_4(u,\Sigma_\alpha)=-\infty\,.
	\]
	By the arbitrariness of $\mu>\overline\mu$, this proves \eqref{eq:levESa} for masses larger than $\overline\mu_\alpha$. 
	
	Conversely, for every $\mu\leq\overline\mu_\alpha$, setting
	\[
	\ee_4^{rad}(\mu,\Sigma_\alpha):=\inf_{\substack{v\in H_\mu^1(\Sigma_\alpha) \\ v \text{ radial} }}E_4(v,\Sigma_\alpha)
	\]
	and exploiting \eqref{eq:equi1} and \eqref{eq:PS1} as in the first part of the proof of Proposition \ref{prop:AGSsector} it is easy to see again that
	\begin{equation}
	\label{eq:ugE}
	\ee_4(\mu,\Sigma_\alpha)=\ee_4^{rad}(\mu,\Sigma_\alpha)=\frac{\alpha}{2\pi}\ee_4\left(\frac{2\pi}{\alpha}\mu,\R^2\right)\,,
	\end{equation}
	that together with \eqref{eq:levER2} and the definition of $\overline\mu_\alpha$ completes the proof of \eqref{eq:levESa} for $\alpha\in(0,\pi]$. Moreover, if $u\in H_{\mu}^1(\Sigma_\alpha)$ is such that $E_4(u,\Sigma_\alpha)=\ee_4(\mu,\Sigma_\alpha)$, then by \eqref{eq:ugE} the same is true for its symmetric rearrangement $u^*$ as in \eqref{eq:u^*1}, and the radial function in the whole $\R^2$ with the same radial profile of $u^*$ is a ground state of $E_4(\cdot,\R^2)$ with mass $2\pi\mu/\alpha$. This implies that $2\pi\mu/\alpha=\overline\mu$, i.e. $\mu=\overline\mu_\alpha$, and that (up to translations when $\alpha=\pi$ and up to a change of sign) $u$ is the restriction of a soliton $\phi_\lambda$ to $\Sigma_\alpha$, for some $\lambda>0$. This proves item (i).
	
	Consider then $\alpha\in(\pi,2\pi)$. Observe that the inequality $\displaystyle \ee_4(\mu,\Sigma_\alpha)\leq\ee_4(\mu,\Sigma_\pi)$ holds for every $\mu>0$, again by the density argument invoked in the proof of Proposition \ref{prop:AGSsector}. Conversely, exploiting \eqref{eq:equi2} and \eqref{eq:PS2} as in the second part of the proof of Proposition \ref{prop:AGSsector} we obtain the reverse inequality $\ee_4(\mu,\Sigma_\alpha)\geq\ee_4(\mu,\Sigma_\pi)$, thus yielding $\ee_4(\mu,\Sigma_\alpha)=\ee_4(\mu,\Sigma_\pi)$ for every $\mu>0$. This completes the proof of \eqref{eq:levESa}. Furthermore, if $u\in H_\mu^1(\Sigma_\alpha)$ is such that $E_4(u,\Sigma_\alpha)=\ee_4(\mu,\Sigma_\alpha)$, then we have again that with no loss of generality $u>0$ on $\Sigma_\alpha$. Hence, the symmetric rearrangement $u^*$ of $u$ on $\Sigma_\pi$ as in \eqref{eq:u^*2} satisfies $u^*\in H_\mu^1(\Sigma_\pi)$ and $E_4(u^*,\Sigma_\pi)<E_4(u,\Sigma_\alpha)=\ee_4(\mu,\Sigma_\alpha)$, the strict inequality being guaranteed by the strict inequality in \eqref{eq:PS2}. Since this is impossible by \eqref{eq:levESa}, we conclude that $\ee_4(\mu,\Sigma_\alpha)$ is not attained for any $\mu>0$. 
\end{proof}

\section{Action ground states on polygons}
\label{sec-actionGS}

This section is devoted to the study of action ground states as defined in \eqref{eq-actiongs} for polygons in $\R^2$. In particular, for the purposes of the present paper we need to develop a detailed analysis of action ground states on polygons, including existence results and information about their mass depending on the actual value of $\lambda$, and a fine asymptotic analysis as $\lambda\to+\infty$ in the case $p=4$. This is done in the next two subsections, in which $\Omega\subset\R^2$ will always denote a given polygon.


\subsection{Existence and qualitative properties}
\label{subsec-exandprop}

Given $\lambda\in\R$, set
\begin{gather}
\label{eq-mmeno} M_p^-(\lambda):=\inf\left\{ \|u\|_{L^2(\Omega)}^2 : u\in\NN_p(\lambda,\Omega),\,J_p(u,\lambda,\Omega)=\JJ_p(\lambda,\Omega) \right\},\\
\label{eq-mpiu} M_p^+(\lambda):=\sup\left\{ \|u\|_{L^2(\Omega)}^2 : u\in\NN_p(\lambda,\Omega),\,J_p(u,\lambda,\Omega)=\JJ_p(\lambda,\Omega) \right\}\,.
\end{gather}
The next results provide a complete portrait of existence of action ground states and relate the quantities above with the action ground state level $\JJ_{p}$.

\begin{proposition}
	\label{agsprop}
	Let $p>2$. Then:
	\begin{enumerate}[label=(\roman*)]
		\item $\JJ_p(\lambda,\Omega)$, $M_p^-(\lambda)$ and $M_p^+(\lambda)$ are attained if and only if $\lambda >0$;
		
		\smallskip
		\item the function $\JJ_p(\cdot,\Omega) : (0,+\infty) \to (0,+\infty)$ is strictly increasing and locally Lipschitz continuous;
		
		\smallskip
		\item for every $\lambda >0$, the left and right derivatives of $\JJ_p(\cdot,\Omega)$ satisfy
		\[
		\JJ_{p,-}'(\lambda,\Omega)=\frac12 M_p^+(\lambda)\qquad\text{and}\qquad \JJ_{p,+}'(\lambda,\Omega)=\frac12 M_p^-(\lambda);
		\]
		
		\smallskip
		\item there exists an at most countable set $Z_p \subset (0,+\infty)$ such that
		\[
		M_p^-(\lambda)=M_p^+(\lambda)=: \mu(\lambda)\qquad\forall\lambda \in (0,+\infty)\setminus Z_p.
		\]
		In particular, $\JJ_p(\cdot,\Omega)$ is differentiable on $(0,+\infty) \setminus Z_p$ with
		\[
		\JJ_p'(\lambda,\Omega) = \frac12 \mu(\lambda)\qquad\forall\lambda \in (0,+\infty)\setminus Z_p.
		\]
	\end{enumerate}
\end{proposition}

\begin{proof}
	Item \emph{(i)} is an immediate consequence of the boundedness of $\Omega$ and classical variational methods. On the other hand, Items \emph{(ii)-(iv)} can be proved arguing as in \cite[proofs of Lemma 2.4 and Theorem 1.5]{DST-23}, since there is no actual difference in these arguments when passing from Dirichlet to Neumann boundary conditions.
\end{proof}

\begin{lemma}
	\label{semic}
	Let $p>2$ and $\lambda>0$. Then:
	\begin{gather}
	\liminf_{q \to p} M_q^-(\lambda) \ge M_p^-(\lambda), \label{Minf} \\
	\limsup_{q \to p} M_q^+(\lambda) \le M_p^+(\lambda). \label{Msup}
	\end{gather}
\end{lemma}

\begin{proof}
	First we show that, for every fixed $\lambda>0$, $\JJ_p(\lambda,\Omega)$ is a continuous function of $p$ on $(2,\infty)$. To this aim it suffices to show that $\JJ_{p_n}(\lambda,\Omega)\to\JJ_{p}(\lambda,\Omega)$, up to subsequences, for every sequence $p_n \to p \in (2,\infty)$ as $n\to+\infty$. Observe that, if $v\in\NN_p(\lambda,\Omega)$ is such that $J_p(v,\lambda,\Omega)=\JJ_p(\lambda,\Omega)$, then setting
	\begin{equation}
	\label{eq-sigmauno}
	\sigma_{\lambda,p_n}(v):=\left(\frac{\|v\|_{L^p(\Omega)}^p}{\|v\|_{L^{p_n}(\Omega)}^{p_n}}\right)^{\frac{1}{p_n-2}},
	\end{equation}
	we obtain that $\sigma_{\lambda,p_n}(v)v\in\NN_{p_n}(\lambda,\Omega)$ and $\displaystyle\lim_n\sigma_{\lambda,p_n}(v)=1$ by dominated convergence, so that
	\begin{align}
	\label{eq-Jbound}
	\limsup_n\JJ_{p_n}(\lambda,\Omega)&\leq\limsup_n J_{p_n}(\sigma_{\lambda,p_n}(v)v,\lambda,\Omega)\nonumber\\
	&=\limsup_n\frac{p_n-2}{2p_n}\sigma_{\lambda,p_n}(v)^{p_n}\|v\|_{L^{p_n}(\Omega)}^{p_n}\nonumber\\
	&=\limsup_n\frac{p_n-2}{2p_n}\|v\|_{L^{p}(\Omega)}^{p}=J_p(v,\lambda,\Omega)=\JJ_p(\lambda,\Omega)
	\end{align}
	Conversely, for each $n$, let $u_n\in\NN_{p_n}(\lambda,\Omega)$ be such that $J_{p_n}(u_n,\lambda,\Omega)=\JJ_{p_n}(\lambda,\Omega)$. By \eqref{eq-Jbound}, $\|u_n\|_{L^{p_n}(\Omega)}^{p_n}$ is uniformly bounded and so, by \eqref{eq-nehari}, $(u_n)_n$ is uniformly bounded in $H^1(\Omega)$. Thus, up to subsequences, $u_n\rightharpoonup u$ in $H^1(\Omega)$ and $u_n\to u$ in $L^q(\Omega)$, for every $q \geq2$, and a.e. on $\Omega$. As a consequence, dominated convergence implies
	\[
	\|u_n\|_{L^{p_n}(\Omega)}^{p_n} \to \|u\|_{L^p(\Omega)}^p\qquad\text{as }n\to+\infty.
	\]
	Note that this immediately gives that $u\not\equiv 0$, since if this were not the case we would have $u_n\to 0$ in $H^1(\Omega)$ as $n\to+\infty$, which is prevented by the fact that the Sobolev inequality and $u_n\in\NN_{p_n}(\lambda,\Omega)$ yield
	\[
	\|u_n\|_{H^1(\Omega)}\geq\left(C_{p_n}\min\{1,\lambda\}\right)^{\frac{1}{p_n-2}},\qquad\text{with}\qquad C_{p_n}:=\inf_{w\in H^1(\Omega)\setminus\{0\}}\frac{\|w\|_{H^1(\Omega)}}{\|w\|_{L^{p_n}(\Omega)}}>c_p>0.
	\]
	Now, setting
	\begin{equation}
	\label{eq-sigmadue}
	\sigma_{\lambda,p}(u_n):=\left(\frac{\|\nabla u_n\|_{L^{2}(\Omega)}^{2}+\lambda\|u_n\|_{L^{2}(\Omega)}^{2}}{\|u_n\|_{L^{p}(\Omega)}^{p}}\right)^{\frac{1}{p-2}}=\left(\frac{\|u_n\|_{L^{p_n}(\Omega)}^{p_n}}{\|u_n\|_{L^{p}(\Omega)}^{p}}\right)^{\frac{1}{p-2}},
	\end{equation}
	we obtain $\sigma_{\lambda,p}(u_n)u_n\in\NN_{p}(\lambda,\Omega)$,  $\displaystyle\lim_n\sigma_{\lambda,p}(u_n)=1$, and
	\begin{multline*}
	\JJ_p(\lambda,\Omega) \le\liminf_n J_{p}(\sigma_{\lambda,p}(u_n)u_n,\lambda,\Omega) = \liminf_n\frac{p-2}{2p}\sigma_{\lambda,p}(u_n)^p\|u_n\|_{L^p(\Omega)}^p\\
	= \liminf_n\frac{p_n-2}{2p_n}\|u_n\|_{L^{p_n}(\Omega)}^{p_n} =\liminf_n J_{p_n}(u_n,\lambda,\Omega)=\liminf_n\JJ_{p_n}(\lambda,\Omega),
	\end{multline*}
	which, combined with \eqref{eq-Jbound}, proves the continuity in $p$ of $\JJ_p(\lambda,\Omega)$.
	
	Furthermore, note that, letting $u_n$ be as above and $u$ be (up to subsequences) its weak limit in $H^1(\Omega)$ as $n\to+\infty$, it is easy to prove that the convergence of $u_n$ to $u$  is actually strong in $H^1(\Omega)$, so that $u\in\NN_p(\lambda,\Omega)$ and $J_p(u,\lambda,\Omega)=\JJ_p(\lambda,\Omega)$. Indeed, we already know that $u_n\to u$ strongly in $L^q(\Omega)$ for every $q\geq2$. If, by contradiction, we do not have $\|\nabla u_n \|_{L^2(\Omega)} \to \|\nabla u\|_{L^2(\Omega)}$, then by \eqref{eq-sigmadue} and the weak lower semicontinuity of the norms it follows
	\[
	\sigma_{\lambda,p}(u) < \liminf_n \sigma_{\lambda,p}(u_n) =1,
	\]
	entailing
	\begin{multline*}
	\JJ_{p}(\lambda,\Omega)\leq J_{p}(\sigma_{\lambda,p}(u)u,\lambda,\Omega) = \frac{p-2}{2p}\sigma_{p,\lambda}(u)^p \|u\|_{L^p(\Omega)}^p =  \sigma_{\lambda,p}^p(u)\lim_n\frac{p_n-2}{2p_n}  \|u_n\|_{L^{p_n}(\Omega)}^{p_n} \\
	= \sigma_{\lambda,p}^p(u)\lim_n J_{p_n}(u_n,\lambda,\Omega) = \sigma_{\lambda,p}^p(u)\lim_n \JJ_{p_n}(\lambda,\Omega)= \sigma_{\lambda,p}^p(u) \JJ_p(\lambda,\Omega) < \JJ_p(\lambda,\Omega),
	\end{multline*}
	that is a contradiction. Hence, up to subsequences $u_n\to u$ in $H^1(\Omega)$ as $n\to+\infty$, so that $u\in\NN_{p}(\lambda,\Omega)$ and
	\[
	J_{p}(u,\lambda,\Omega) = \frac{p-2}{2p} \|u\|_{L^p(\Omega)}^p = \lim_n\frac{p_n-2}{2p_n}  \|u_n\|_{L^{p_n}(\Omega)}^{p_n}= \lim_n \JJ_{p_n}(\lambda,\Omega) = \JJ_p(\lambda,\Omega).
	\]
	Combining all the previous results, \eqref{Minf} and \eqref{Msup} follow by the definition of $M_p^{\pm}(\lambda)$ and the fact that they are attained.
\end{proof}


\subsection{Asymptotics for large $\lambda$ in the $L^2$-critical case}
\label{subsec-asymptotics}

Here we discuss the asymptotic behavior of $\JJ_4(\lambda,\Omega)$ and of the mass of the associated action ground states for large $\lambda$. 

\begin{proposition}
	\label{asymp}
	Let $\overline\alpha\in(0,\pi)$ be the amplitude of the smallest internal angle of $\Omega$, and $\overline\mu_{\overline\alpha}$ be the corresponding value defined in \eqref{eq:mua}. Then:
	\begin{enumerate}[label=(\roman*)]
		\item $\displaystyle\lim_{\lambda\to+\infty}\big( \JJ_4(\lambda,\Omega) - \JJ_4(\lambda,\Sigma_{\overline\alpha})\big) = 0$, with $\Sigma_{\overline\alpha}$ the set corresponding to $\overline\alpha$ as in \eqref{eq-Calpha};
		
		\smallskip
		\item for every family $(u_\lambda)_{\lambda>0}$ such that $u_\lambda\in\NN_4(\lambda,\Omega)$ and $J_4(u_\lambda,\lambda,\Omega)=\JJ_4(\lambda,\Omega)$, there results
		\[
		\lim_{\lambda\to+\infty}\|u_\lambda\|_{L^2(\Omega)}^2=\overline\mu_{\overline\alpha}.
		\]
	\end{enumerate}
\end{proposition}

\begin{proof}
	We divide the proof in three steps. The combination of the first two proves \emph{(i)}, while the last proves \emph{(ii)}. Also, note that throughout the proof it is understood that the limits are for $\lambda\to+\infty$ (we omit to repeat it).
	
	\smallskip
	\emph{Step 1): proof of $\JJ_4(\lambda,\Omega)\leq\JJ_4(\lambda,\Sigma_{\overline\alpha})+o(1)$}. Recall that, by Proposition \ref{prop:AGSsector}, the unique positive radial action ground state in $\NN_4(\lambda,\Sigma_{\overline\alpha})$ is the restriction to $\Sigma_{\overline\alpha}$ of the soliton $\phi_\lambda$ of $\R^2$.  Since $\phi_\lambda$ is given by \eqref{eq:phil} and the radial profile of $\phi_1$ decreases exponentially fast as $|x|\to+\infty$ (see, e.g., \cite{BL-83}), $\phi_\lambda(R)$ is exponentially small as $\lambda\to+\infty$, for every given $R>0$. 
	
	With no loss of generality let the corner point of $\partial\Omega$ corresponding to the angle $\overline\alpha$ be fixed in the origin of $\R^2$ in such a way that, locally around this point, $\Omega$ can be seen as a subset of the sector $\Sigma_{\overline\alpha}$. Take then $R>0$ such that the closure of $\Omega\cap B_R(0)$ contains no corner point of $\partial\Omega$ other than the origin (here $B_R(0)$ denotes the ball of $\R^2$ with radius $R$ centered at the origin) and define $v_\lambda\in H^1(\Omega)$ as the restriction to $\Omega$ of $(\phi_\lambda-\phi_\lambda(R))_+$. The exponential decay of $\phi_\lambda(R)$ ensures that for every $q\geq2$ there exists $\gamma_q>0$ so that
	\[
	\|v_\lambda\|_{L^q(\Omega)}^q=\|\phi_\lambda\|_{L^q(\Sigma_{\overline\alpha})}^q+o(e^{-\lambda^{\gamma_q}})\,,
	\]
	and since $\|\nabla v_\lambda\|_{L^2(\Omega)}\leq\|\nabla\phi_\lambda\|_{L^2(\Sigma_{\overline\alpha})}$ by construction, it follows that
	\[
	\begin{split}
	\sigma_\lambda(v_\lambda):=&\,\left(\frac{\|\nabla v_\lambda\|_{L^2(\Omega)}^2+\lambda\|v_\lambda\|_{L^2(\Omega)}^2}{\|v_\lambda\|_{L^4(\Omega)}^4}\right)^{1/2}\\
	\leq&\, \left(\frac{\|\nabla \phi_\lambda\|_{L^2(\Sigma_{\overline\alpha})}^2+\lambda\|\phi_\lambda\|_{L^2(\Sigma_{\overline\alpha})}^2}{\|\phi_\lambda\|_{L^4(\Sigma_{\overline\alpha})}^4+o(e^{-\lambda^{\gamma_4}})}\right)^{1/2}=1+o(e^{-\lambda^{\gamma_4}})\,.
	\end{split}
	\]
	Since $\sigma_\lambda(v_\lambda)v_\lambda\in\NN_4(\lambda,\Omega)$, we obtain
	\[
	\begin{split}
	\JJ_4(\lambda,\Omega)\leq J_4(\sigma_\lambda(v_\lambda)v_\lambda,\lambda,\Omega)=&\,\frac14\sigma_\lambda(v_\lambda)^4\|v_\lambda\|_{L^4(\Omega)}^4\\
	\leq&\,\frac14(1+o(e^{-\lambda^{\gamma_4}}))\|\phi_\lambda\|_{L^4(\Sigma_{\overline\alpha})}^4=\JJ_4(\lambda,\Sigma_{\overline\alpha})+o\big(e^{-\lambda^{\frac{\gamma_4}2}}\big)\,,
	\end{split}
	\]
	the last equality being justified by Proposition \ref{prop:AGSsector} and the fact that $\JJ_4(\lambda,\R^2)=\lambda\overline\mu/2$ for every $\lambda>0$ (see, e.g., \cite[Proposition 2.3]{DST-23}), so that $\|\phi_\lambda\|_{L^4(\Sigma_{\overline\alpha})}^4$ grows linearly with $\lambda$ as $\lambda\to+\infty$.
	
	\smallskip
	\emph{Step 2): proof of $\JJ_4(\lambda,\Omega)\geq\JJ_4(\lambda,\Sigma_{\overline\alpha})+o(1)$}. Let $u_\lambda\in\NN_4(\lambda,\Omega)$ be a positive function such that $J_4(u_\lambda,\lambda,\Omega)=\JJ_4(\lambda,\Omega)$. Arguing as in \cite[Theorem 2.1]{NT1} it can be shown that, for $\lambda$ large enough, there exists a connected set $\Omega_\lambda\subset\Omega$ with $\text{\normalfont diam}\big(\Omega_\lambda\big)\leq C\lambda^{-1/2}$, for a suitable constant $C>0$ independent of $\lambda$, such that $\Omega_\lambda$ contains every local maximum point of $u_\lambda$. Furthermore, exploiting Proposition \ref{prop:AGSsector} and Step 1 of the proof, as $\lambda$ increases one obtains that $\Omega_\lambda$ concentrates around the origin (that is the vertex of $\partial\Omega$ corresponding to the angle with amplitude $\overline\alpha$). Hence, letting $A_\lambda:=\Omega\cap B_{\lambda^{-1/4}}(0)$ and following again the argument in \cite[Theorem 2.3]{NT1}, one has
	\begin{equation}
	\label{eq:exp}
	u_\lambda(x)\leq C_1\sqrt\lambda e^{-c_2\lambda^{1/4}}\qquad\forall x\in\Omega\setminus A_\lambda\,,
	\end{equation}
	with $C_1,C_2>0$ independent of $\lambda$.
	
	Set then $\displaystyle\varepsilon_\lambda:=\max_{x\in\overline{\Omega\setminus A_\lambda}}u_\lambda(x)$ and $v_\lambda:=(u_\lambda-\varepsilon_\lambda)_+$. By construction, $v_\lambda\equiv0$ on $\Omega\setminus A_\lambda$, so that we can think of $v_\lambda$ as a compactly supported function on the sector $\Sigma_{\overline\alpha}$. Furthermore, since $\varepsilon_\lambda$ is exponentially small as $\lambda\to+\infty$, while $\|u_\lambda\|_{L^4(\Omega)}$ is at most linear in $\lambda$ by the estimate obtained in Step 1, direct computations yield
	\[
	\|v_\lambda\|_{L^4(\Sigma_{\overline\alpha})}^4=\|u_\lambda\|_{L^4(\Omega)}^4+r_\lambda
	\]
	with $r_\lambda$ exponentially small as $\lambda\to+\infty$. Since by construction we also have that $\|\nabla v_\lambda\|_{L^2(\Sigma_{\overline\alpha})}\leq\|\nabla u_\lambda\|_{L^2(\Omega)}$, $\|v_\lambda\|_{L^2(\Sigma_{\overline\alpha})}\leq\|u_\lambda\|_{L^2(\Omega)}$, it follows that
	\begin{equation}
	\label{sigma}
	\sigma_\lambda(v_\lambda):=\left(\frac{\|\nabla v_\lambda\|_{L^2(\Sigma_{\overline\alpha})}^2+\lambda\|v_\lambda\|_{L^2(\Sigma_{\overline\alpha})}^2}{\|v_\lambda\|_{L^4(\Sigma_{\overline\alpha})}^4}\right)^{1/2}\leq \left(\frac{\|\nabla u_\lambda\|_{L^2(\Omega)}^2+\lambda\|u_\lambda\|_{L^2(\Omega)}^2}{\|u_\lambda\|_{L^4(\Omega)}^4+r_\lambda}\right)^{1/2}=1+\widetilde{r}_\lambda\,,
	\end{equation}
	with $\widetilde{r}_\lambda$ exponentially small as $\lambda\to+\infty$, in turn entailing
	\[
	\JJ_4(\lambda,\Sigma_{\overline\alpha})\leq J_4(\sigma_\lambda(v_\lambda)v_\lambda,\lambda,\Sigma_{\overline\alpha})=\frac14\sigma_\lambda(v_\lambda)^4\|v_\lambda\|_{L^4(\Sigma_{\overline\alpha})}^4\leq \frac14\|u_\lambda\|_{L^4(\Omega)}^4+\widehat{r}_\lambda=\JJ_4(\lambda,\Omega)+\widehat{r}_\lambda\,,
	\]
	with $\widehat{r}_\lambda$ exponentially small as $\lambda\to+\infty$, that is the estimate we were seeking.

	\smallskip
	\emph{Step 3): proof of (ii).} 
	Let $u_\lambda$ be an action ground state in $\NN_4(\lambda, \Omega)$ and define again the function $v_\lambda\in H^1(\Sigma_{\overline\alpha})$ as in Step 2 above. Moreover, set
	\[
	w_\lambda:=\sigma_\lambda(v_\lambda)v_\lambda
	\]
	with $\sigma_\lambda(v_\lambda)$ as in \eqref{sigma}, so that $w_\lambda\in \NN_4(\lambda,\Sigma_{\overline\alpha})$ for every $\lambda$ and $J_4(w_\lambda,\lambda,\Sigma_{\overline\alpha})\leq \JJ_4(\lambda,\Sigma_{\overline\alpha})+\widehat{r}_\lambda$ for sufficiently large $\lambda$. As a consequence, the function $\overline w_\lambda(x): =\lambda^{-1/2} w_\lambda\big(x/\sqrt\lambda\big)$ belongs to $\NN_4(1,\Sigma_{\overline\alpha})$ and (by (i), Proposition \ref{prop:AGSsector} and the explicit expression of $\JJ_4(\lambda,\R^2)$) 
	\begin{align*}
	\JJ_4(1,\Sigma_{\overline\alpha})\leq J_4(\overline w_\lambda,1, \Sigma_{\overline\alpha}) &= \frac1\lambda J_4\left(w_\lambda ,\lambda,\Sigma_{\overline\alpha}\right)
	\leq \frac1\lambda\JJ_4(\lambda,\Sigma_{\overline\alpha})+\frac{\widehat{r}_\lambda}\lambda    \\
	& =  \frac1\lambda \frac{\overline\alpha}{2\pi}\JJ_4(\lambda,\R^2) + o(1) = \frac12\frac{\overline\alpha}{2\pi}\overline\mu +o(1)=\JJ_4(1,\Sigma_{\overline\alpha})+o(1).
	\end{align*}
	This shows that $(\overline w_\lambda)_{\lambda>0}$ is a minimizing sequence for $J_4(\cdot,1,\Sigma_{\overline\alpha})$ on $\NN_4(1,\Sigma_{\overline\alpha})$, and Proposition \ref{prop:AGSsector} then ensures that (up to a change of sign) $\overline w_\lambda$ converges strongly in $H^1(\Sigma_{\overline\alpha})$ to the restriction to $\Sigma_{\overline\alpha}$ of the soliton $\phi_\lambda$. Hence, $\|\overline w_\lambda\|_{L^2(\Sigma_{\overline\alpha})}^2=\|\phi_\lambda\|_{L^2(\Sigma_{\overline\alpha})}^2+o(1)=\overline\mu_\alpha+o(1)$, and a direct computation shows that the same is true for $\|w_\lambda\|_{L^2(\Sigma_{\overline\alpha})}^2$ too. Since the final computation of Step 2 above shows that $\sigma_\lambda(v_\lambda)\to1$ as $\lambda\to+\infty$, the asymptotic behaviour of $\|w_\lambda\|_{L^2(\Sigma_{\overline\alpha})}^2$ coincides with that of $\|v_\lambda\|_{L^2(\Sigma_{\overline\alpha})}^2$, and thus also with that of $\|u_\lambda\|_{L^2(\Omega)}$, since they differ by a correction that is exponentially small as $\lambda$ diverges to infinity. As this completes the proof of (ii), we conclude.
\end{proof}

\section{$L^2$-critical energy ground states on polygons}
\label{sec-energyGS}

The aim of this section is to discuss existence and nonexistence of energy ground states on polygons in the $L^2$-critical case, i.e. problem \eqref{eq-energygs} with $p=4$. Precisely, the main result of the section is the following proposition.

\begin{proposition}
	\label{egs}
	Let $\overline\alpha\in(0,\pi)$ be the amplitude of the smallest internal angle of  $\Omega$, and $\overline\mu_{\overline\alpha}$ be the corresponding value defined in \eqref{eq:mua}. Then
	\[
	\ee_4(\mu,\Omega)\begin{cases}
	<0 & \text{if }\mu\leq\overline\mu_{\overline\alpha}\\
	=-\infty & \text{if }\mu>\overline\mu_{\overline\alpha}
	\end{cases}
	\]
	and it is attained if and only if $\mu\in(0,\overline\mu_{\overline\alpha}]$.
\end{proposition}
The proof of Proposition \ref{egs} is based on a detailed compactness analysis.
We begin with two technical lemmas, the first of which is a variant of Lemma 2.18 in \cite{CZR} or Lemma I.1 in \cite{PL2}. Even though both lemmas are rather classical, we provide the proof for the sake of completeness.

In the following, given a polygon $\Omega\subset\R^2$, a sequence of points $(x_n)_n\subset\overline\Omega$, and a sequence of positive numbers $(\delta_n)_n$ such that $\delta_n\to0$ as $n\to\infty$, we will consider the sets
\begin{equation}
\label{eq:On}
\Omega_n: = \frac{\Omega - x_n}{\delta_n} = \left\{y \in \R^2 \mid y = (x - x_n)/\delta_n \text{ for some } x \in \Omega\right\}.
\end{equation}
Observe that, depending on $d(x_n,\partial\Omega)/\delta_n$ being uniformly bounded from above or not, $\Omega_n$ tends to $\Omega_\infty$ as $n\to\infty$, with (up to rotations) $\Omega_\infty$ being an infinite sector $\Sigma_\alpha$, with $\alpha$ equal to $\pi$ or to one of the internal angles of $\Omega$, or the whole plane $\R^2$ (here $d(x_n,\partial\Omega)$ denotes the distance of $x_n$ from the boundary of $\Omega$).

\begin{lemma}
\label{CZR}
Let $\Omegan$ be as in \eqref{eq:On} and let $w_n \in H^1(\Omegan)$ satisfy $\sup_n \|w_n\|_{ H^1(\Omegan)} <+\infty$. If there exists $R>0$ such that 
\[
\liminf_n \max_{y\in \overline\Omega_n} \|w_n\|_{L^2(B_R(y)\cap\Omegan)} = 0, 
\]
then, up to subsequences, 
\[
\liminf_n \|w_n\|_{L^4(\Omegan)} = 0.
\]
\end{lemma}

\begin{proof} Cover $\R^2$ with a family $\{B_R(y_j)\}_{j\in\Nat}$ of balls of radius $R$ in such a way that every point of $\R^2$ belongs to at most $m$ balls, for some constant $m\in\mathbb N$, and, for each $n$, let $J_n\subset \Nat$ be the set of indices $j$ for which $y_j \in \overline \Omega_n$. Note that, since $\Omega$ is a polygon, $R$ is fixed and $\delta_n \le 1$ for every $n$, the best constant in the Gagliardo-Nirenberg inequality \eqref{GN-bounded} with $p=4$ on $\Omega_n \cap B_R(y_j)$ is uniformly bounded in $j$ and $n$ by a constant that we denote by $K_R$. Then we can estimate
\begin{align*}
\|w_n\|_{L^4(\Omegan)}^4 &\le \sum_{j \in J_n} \ \|w_n\|_{L^4(B_R(y_j)\cap\Omegan)}^4 \le
\sum_{j \in J_n} K_R \|w_n\|_{L^2(B_R(y_j)\cap\Omegan)}^2 \|w_n\|_{H^1(B_R(y_j)\cap\Omegan)}^2\\
& \le K_R \max_{y \in \overline \Omega_n}  \|w_n\|_{L^2(B_R(y)\cap\Omegan)}^2 \sum_{j \in J_n} \|w_n\|_{H^1(B_R(y_j)\cap\Omegan)}^2\\
&\le K_R m \|w_n\|_{H^1(\Omegan)}^2 \max_{y \in \overline\Omega_n} \|w_n\|_{L^2(B_R(y)\cap\Omegan)}^2 \\
&\le C K_R m \max_{y \in \overline\Omega_n} \|w_n\|_{L^2(B_R(y)\cap\Omegan)}^2,
\end{align*}
a quantity that, up to subsequences, tends to $0$ as $n \to \infty$ by assumption.
\end{proof}

\begin{remark}
\label{estimate}
The preceding proof shows, in particular, that
\begin{equation}
\label{bound1}
\max_{y \in \overline\Omega_n} \|w_n\|_{L^2(B_R(y)\cap\Omegan)}^2 \ge \frac{\|w_n\|_{L^4(\Omegan)}^4}{ K_R m \|w_n\|_{H^1(\Omega_n)}^2 }\,,
\end{equation}
an estimate that we will use later on.
\end{remark}

\begin{lemma}
\label{BLn}
Let $\Omega_n$ be as in \eqref{eq:On}, $\Omega_\infty$ its corresponding limit set as above, and  $u_n \in  H^1(\Omega_n)$ such that 
\begin{equation}
\label{loc}
\sup_n \|u_n\|_{H^1(\Omega_n)} < +\infty\,. 
\end{equation}
Then there exist a subsequence of $u_n$ (not relabelled) and a function $u \in H^1(\Omega_\infty)$ such that, as $n \to \infty$, $u_n \rightharpoonup u$ weakly in $H^1(K)$ and $u_n\to u$ strongly in $L^p(K)$, for every $p\geq 2$ and every compact subset $K$ of $\Omega_\infty$, and

\begin{align}
&\int_{\Omega_n} |u_n|^p\dx = \int_{\Omega_n} |u_n-u|^p\dx +\int_{\Omega_\infty} |u|^p\dx + o(1) \qquad \forall p\geq2 \label {bl1},\\
&\int_{\Omega_n} |\nabla u_n|^2\dx = \int_{\Omega_n} |\nabla u_n-\nabla u|^2\dx +\int_{\Omega_\infty} |\nabla u|^2\dx + o(1) \label {bl2}.
\end{align}
\end{lemma}

\begin{proof} Extend each $u_n$ to $\Omega_\infty$ by defining
\[
\overline u_n(x)  = \begin{cases} u_n(x) & \text{ if } x\in \Omega_n \\ 0  &\text{ if } x\in \Omega_\infty \setminus\Omega_n.\end{cases}
\]
Since $\overline u_n$ is bounded in $L^2(\Omega_\infty)$, (a subsequence of) $\overline u_n$ converges weakly in $L^2(\Omega_\infty)$ to some $u$.
By definition of $\Omega_n$ and $\Omega_\infty$, every compact set $K \subset \Omega_\infty$ is contained in $\Omega_n$ for every $n$ large. Therefore, by \eqref{loc}, $u_n$ converges to $u$ weakly in $H^1(K)$, strongly in $L^p(K)$ for all $p\geq 2$, and a.e. on $\Omega_\infty$, and it follows easily that $u \in H^1(\Omega_\infty)$.

By the Brezis-Lieb Lemma \cite{BL83}, for every $p\geq2$, as $n \to \infty$,
\[
\int_{\Omega_\infty} |\overline u_n |^p\dx = \int_{\Omega_\infty} |\overline u_n -u|^p\dx +\int_{\Omega_\infty} | u |^p\dx +o(1),
\] 
namely
\begin{align*}
\int_{\Omega_n} |u_n |^p\dx &= \int_{\Omega_n} | u_n -u|^p\dx +  \int_{\Omega_\infty\setminus \Omega_n} |u|^p\dx+ \int_{\Omega_\infty} | u |^p\dx +o(1)\\
&= \int_{\Omega_n} | u_n -u|^p\dx +  \int_{\Omega_\infty} | u |^p\dx +o(1),
\end{align*}
since $u$ is fixed and $\Omega_n \to \Omega_\infty$. This proves \eqref{bl1}. To prove  \eqref{bl2}, define
\[
\overline{\nabla u}_n(x)  = \begin{cases} \nabla u_n(x) & \text{ if } x\in \Omega_n \\ 0  &\text{ if } x\in \Omega_\infty \setminus\Omega_n.\end{cases}
\]
and note that $\overline{\nabla u}_n \to \nabla u$ in the sense of distributions in $\Omega_\infty$. Since $C^\infty_0(\Omega_\infty)$ is dense in $L^2(\Omega_\infty)$, actually $\overline{\nabla u}_n \rightharpoonup \nabla u$ weakly in $L^2(\Omega_\infty)$. Then, as $n\to \infty$,
\[
\int_{\Omega_\infty} \left(\overline{\nabla u}_n - \nabla u\right)\cdot \nabla u \dx \to 0
\]
from which we obtain
\begin{align*}
o(1) &= \int_{\Omega_\infty} \left(\overline{\nabla u}_n - \nabla u\right)\cdot \nabla u \dx = \int_{\Omega_n} \left(\nabla u_n - \nabla u\right)\cdot \nabla u \dx -\int_{\Omega_\infty\setminus\Omega_n}  |\nabla u|^2 \dx\\
&= \int_{\Omega_n} \left(\nabla u_n - \nabla u\right)\cdot \nabla u \dx +o(1)
\end{align*}
since, as above, $u$ is fixed and $\Omega_n \to \Omega_\infty$. This shows that 
\[
 \int_{\Omega_n} \nabla u_n \cdot \nabla u \dx =  \int_{\Omega_n} |\nabla u|^2  \dx + o(1)
\]
in turn yielding
\begin{align*}
\int_{\Omega_n} |\nabla u_n -\nabla u|^2\dx &= \int_{\Omega_n} |\nabla u_n|^2dx + \int_{\Omega_n} |\nabla u|^2\dx -2\int_{\Omega_n} \nabla u_n \cdot\nabla u\dx \\
&=  \int_{\Omega_n} |\nabla u_n|^2dx - \int_{\Omega_n} |\nabla u|^2\dx +o(1),
\end{align*}
which is \eqref{bl2} rearranging terms and observing that $\int_{\Omega_n} |\nabla u|^2\dx = \int_{\Omega_\infty} |\nabla u|^2\dx +o(1)$.
\end{proof}

Since from now we only deal with the critical case $p = 4$, we drop the subscript $4$ in symbols like $E_4(u_n,\Omega), \ee_4(\mu,\Omega)$ and so on.

To prove Proposition \ref{egs} we address separately the case $\mu\in(0,\overline\mu_{\overline\alpha})$ and $\mu=\overline\mu_{\overline\alpha}$. The next proposition and its corollary deal with the former regime.
\begin{proposition} 
\label{boundbelow}
Let $\overline\alpha\in(0,\pi)$ be the amplitude of the smallest internal angle of $\Omega$, and $\overline\mu_{\overline\alpha}$ be the corresponding value defined in \eqref{eq:mua}. For every $\mu \in (0,\overline\mu_{\overline\alpha})$ there results
\[
\ee(\mu,\Omega) > -\infty.
\]
\end{proposition}

\begin{proof} Assume by contradiction that $\ee(\mu,\Omega) = -\infty$ and let $(u_n)_n \subset H^1_\mu(\Omega)$ satisfy
$E(u_n,\Omega) \to -\infty$ as $n\to \infty$, so that $\|u_n\|_{L^4(\Omega)} \to \infty$ and, by \eqref{GN-bounded},
$\|\nabla u_n\|_{L^2(\Omega)} \to \infty$, since $\|u_n\|_{L^2(\Omega)}^2 = \mu$ for every $n$.

Set $\delta_n := \|\nabla u_n\|_{L^2(\Omega)}^{-1} \to 0$ and define $v_n(x) := \delta_n u_n(\delta_n x)$ for every $x\in\Omega/\delta_n$. By definition, the functions $v_n$ satisfy
\begin{align*}
&\|v_n\|_{L^2(\Omega/\delta_n)}^2 = \|u_n\|_{L^2(\Omega)}^2 = \mu \\
&\|\nabla v_n\|_{L^2(\Omega/\delta_n)}^2 = \delta_n^2 \|\nabla u_n\|_{L^2(\Omega)}^2 = 1 \\
& \|v_n\|_{L^4(\Omega/\delta_n)}^4 =  \delta_n^2 \|u_n\|_{L^4(\Omega)}^4,
\end{align*}
so that, for every $n$ large,
\begin{equation}
\label{Eneg}
E(v_n,\Omega/\delta_n) = \delta_n^2 E(u_n, \Omega) \le 0.
\end{equation}
We assume for simplicity that the preceding inequality holds for every $n$. Note that this implies that
\begin{equation}
\label{nega}
1 = \|\nabla v_n\|_{L^2(\Omega/\delta_n)}^2 \le \frac12  \| v_n\|_{L^4(\Omega/\delta_n)}^4 
\end{equation}
for every $n$, namely, the $L^4$ norm of $v_n$ is bounded away from zero. Lemma \ref{CZR} (and Remark \ref{estimate}) applied to $v_n$ (with $x_n =0$) shows that there exist $R>0$ and a sequence $y_n^1\in \overline\Omega/\delta_n$ such that 
\[
\liminf_n \|v_n\|_{L^2(B_R(y_n^1) \cap \Omega/\delta_n)}^2 \ge \liminf_n \frac{ \|v_n\|_{L^4(\Omega/\delta_n)}^4}{K_R m\|v_n\|_{H^1(\Omega/\delta_n)}^2}
\ge\frac{2}{K_R m(1+\mu)} =: \beta_0>0,
\]
since $\|v_n\|_{H^1(\Omega/\delta_n)}^2 = 1+\mu$ and by \eqref{nega}.

Set $v^1_n(x) = v_n(x+y_n^1)$. This function is defined in $\Omega^1_n := \Omega/\delta_n - y_n=(\Omega-x_n^1)/\delta_n$ with $x_n^1:=\delta_n y_n^1$, and satisfies
\[
\|\nabla v^1_n \|_{L^2(\Omega_n^1)} = \|\nabla v_n \|_{L^2(\Omega/\delta_n)}, \quad \|v^1_n \|_{L^2(\Omega_n^1)} = \|v_n \|_{L^2(\Omega/\delta_n)}, \quad
\|v^1_n \|_{L^4(\Omega_n^1)} = \|v_n \|_{L^4(\Omega/\delta_n)}.
\]
Moreover,
\begin{equation}
\label{nozero}
\liminf_n \|v^1_n\|_{L^2(B_R(0) \cap \Omega_n^1)}^2 \ge \beta_0.
\end{equation}
Furthermore, as pointed out before Lemma \ref{CZR}, the sets $\Omega_n^1$ tend to a set $\Omega_\infty^1$ which is either $\R^2$ or a sector $\Sigma_\alpha$ with $\alpha\geq\overline\alpha$, depending on the distance of the points $y_n^1$ from $\partial(\Omega/\delta_n)$.
Hence, by Lemma \ref{BLn} there exists $v^1\in H^1(\Omega_\infty^1)$ such that, as $n\to \infty$,
\begin{equation}
\label{BL1}
E(v_n^1,\Omega_n^1) = E(v_n^1 -v^1,\Omega_n^1) + E(v^1,\Omega_\infty^1) +o(1).
\end{equation}
Notice that $v^1$ does not vanish identically, because $ \|v^1\|_{L^2(B_R(0) \cap \Omega_\infty^1)}^2 \ge \beta_0$ by strong convergence on compact sets, and that  $E(v^1,\Omega_\infty^1) >0$ by Proposition \ref{prop:EGSsector}, because $v^1\not\equiv 0$ and $\|v^1\|_{L^2(\Omega_\infty^1)}^2 \le \mu < \overline\mu_{\overline\alpha}$. Moreover, again  by Lemma \ref{BLn},
\begin{equation}
\mu =   \|v_n^1\|_{L^2(\Omega^1_n)}^2 =  \|v_n^1 -v^1\|_{L^2(\Omega^1_n)}^2  +  \|v^1\|_{L^2(\Omega_\infty^1)}^2 +o(1) \ge  \|v_n^1 -v^1\|_{L^2(\Omega^1_n)}^2  + \beta_0+o(1). \label{mass0}
\end{equation}
Now,  by \eqref{Eneg} and \eqref{BL1},
\begin{align}
0 &\ge E(v_n,\Omega/\delta_n) = E(v_n^1,\Omega_n^1) =\nonumber \\
&= \frac12 \|\nabla v_n^1 -\nabla v^1 \|_{L^2(\Omega_n^1)}^2  -  \frac14 \| v_n^1 - v^1 \|_{L^4(\Omega_n^1)}^4 + E(v^1,\Omega_\infty^1) + o(1) \nonumber \\
&\ge -  \frac14 \| v_n^1 - v^1 \|_{L^4(\Omega_n^1)}^4   + E(v^1,\Omega_\infty^1) + o(1). \label{split0}
\end{align}
If $\displaystyle\liminf_n  \|v_n^1 -v^1\|_{L^4(\Omega_n^1)}^4< 4  E(v^1,\Omega_\infty^1)$, we reach a contradiction, and the proposition is proved. Otherwise,
\[
\liminf_n \|v_n^1 -v^1\|_{L^4(\Omega_n^1)}^4 \ge 4 E(v^1,\Omega_\infty^1) >0.
\]
Lemma \ref{CZR} applied to $v_n^1 -v^1$ shows that there exist $y_n^2\in \Omega_n^1$ such that 
\begin{equation}
\label{massbound}
\liminf_n \|v_n^1 -v^1\|_{L^2(B_{R}(y_n^2) \cap \Omega_n^1)}^2 \ge \frac{2  E(v^1,\Omega_\infty^1)}{K_{R} m(1+\mu)} =:\beta_1.
\end{equation}
Defining $v_n^2(x) = v_n^1(x+y_n^2)- v^1(x+y_n^2)$ on $\Omega_n^2 :=\Omega_n^1 -y_n^2=(\Omega-x_n^2)/\delta_n$, with $x_n^2=x_n^1+\delta_n y_n^2$, and working exactly like we did above for $v^1_n$ we see that there exists a nonzero $v^2$ defined on the limit $\Omega^2_\infty$ of $\Omega_n^2$ such that, as $n\to\infty$,
\begin{equation}
\label{BL2}
E(v_n^2,\Omega_n^2) = E(v_n^2 -v^2,\Omega_n^2) + E(v^2,\Omega_\infty^2) +o(1) 
\end{equation}
and
\[
 \|v_n^2\|_{L^2(\Omega^2_n)}^2 =  \|v_n^2 -v^2\|_{L^2(\Omega^2_n)}^2  +  \|v^2\|_{L^2(\Omega_\infty^2)}^2 +o(1).
\]
Therefore we can update \eqref{mass0}  and \eqref{split0}  by writing
\[
\mu \ge  \|v_n^1 -v^1\|_{L^2(\Omega^1_n)}^2  + \beta_0 +o(1) =   \|v_n^2\|_{L^2(\Omega^2_n)}^2+ \beta_0 +o(1) \ge 
  \|v_n^2 -v^2\|_{L^2(\Omega^2_n)}^2 +  \beta_1 + \beta_0 +o(1)
\]
and
\begin{align*}
0 &\ge E(v^1_n,\Omega^1_n) = E(v_n^1 -v^1,\Omega_n^1) + E(v^1,\Omega_\infty^1) +o(1) = E(v_n^2,\Omega_n^2) + E(v^1,\Omega_\infty^1) +o(1) \\
& =  E(v_n^2 -v^2,\Omega_n^2) + E(v^2,\Omega_\infty^2)  + E(v^1,\Omega_\infty^1) +o(1) \\
&\ge  -\frac14  \| v_n^2 - v^2 \|_{L^4(\Omega_n^2)}^4 + E(v^2,\Omega_\infty^2)  + E(v^1,\Omega_\infty^1) +o(1),
\end{align*}
where the term $ E(v^2,\Omega_\infty^2)$ is positive for the same reasons why $ E(v^1,\Omega_\infty^1)$ was positive. If $\displaystyle\liminf_n  \|v_n^2 -v^2\|_{L^4(\Omega_n^1)}^4<  4  E(v^2,\Omega_\infty^2)+ 4  E(v^1,\Omega_\infty^1)$, we again reach a contradiction. Otherwise 
\[
\liminf_n  \|v_n^2 -v^2\|_{L^4(\Omega_n^1)}^4 \ge 4 E(v^2,\Omega_\infty^2)+ 4 E(v^1,\Omega_\infty^1) > 4 E(v^1,\Omega_\infty^1)
\]
and we proceed with $v_n^2 -v^2$ exactly as we did for $v_n^1 -v^1$. At each step the preceding bound applies to $v_n^k-v^k$  and  \eqref{massbound}  yields
\[
\liminf_n \|v_n^k -v^k\|_{L^2(B_{R_k}(y_n^k) \cap \Omega_n^k)}^2 \ge \beta_1.
\]
Therefore, after $k$ steps we obtain
\[
\mu \ge  \|v_n^{k+1} -v^{k+1}\|_{L^2(\Omega^{k+1}_n)}^2  +  k\beta_1 + \beta_0 +o(1) 
\]
so the iterations must stop at some $k$, meaning that 
\[
\liminf_n \|v_n^k -v^k\|_{L^4(\Omega^{k+1}_n)}^4 < 4 E(v^k,\Omega^k_\infty) + \dots + 4 E(v^1,\Omega^1_\infty)
\]
and reaching a contradiction as described above.
\end{proof}

\begin{corollary}
\label{attained0}
For every $\mu \in (0,\overline\mu_{\overline\alpha})$ there exists $u \in H^1_\mu(\Omega)$ such that $E(u,\Omega) =  \ee(\mu,\Omega)$.
\end{corollary}

\begin{proof} First notice that by Proposition \ref{boundbelow}
\begin{equation}
\label{negneg}
-\infty < \ee(\mu,\Omega) \le E\left(\sqrt{\mu/|\Omega|},\Omega\right) < 0.
\end{equation}
Let $(u_n)_n \in H^1_\mu(\Omega)$ be a minimizing sequence for $\ee(\mu,\Omega)$ and assume by contradiction that $\|\nabla u_n\|_{L^2(\Omega)} \to \infty$. Then $E(u_n,\Omega) \le  0$ (for $n$ large) implies  $\|u_n\|_{L^4(\Omega)} \to \infty$. This  feature is enough to repeat the proof of Proposition \ref{boundbelow} (the assumption $\ee(\mu,\Omega) = -\infty$ is used there only to deduce that the $L^4$ part of $E$ tends to $+\infty$). The argument of Proposition \ref{boundbelow} leads then to $\ee(\mu,\Omega) \ge0$, contradicting \eqref{negneg}. Therefore the sequence $(u_n)_n$ is bounded in $H^1(\Omega)$ and by the boundedness of $\Omega$ the thesis follows from standard compactness arguments.
\end{proof}

To complete the proof of Proposition \ref{egs} we are left to consider the case $\mu=\overline\mu_{\overline\alpha}$. Note that the only difference with respect to the argument in the proof of Proposition \ref{boundbelow} lies in the fact that the energy of the weak limit of a minimizing sequence can be equal to zero. However, by Proposition \ref{prop:EGSsector} this can occur if and only if the weak limit is a ground state for $E$ on $\Sigma_{\overline\alpha}$ with mass $\overline\mu_{\overline\alpha}$. To rule out this possible loss of compactness, we now exploit ground states with masses in a left neighbourhood of $\overline\mu_{\overline\alpha}$ (that we already know to exist by Corollary \ref{attained0} above) and we show that, up to subsequences, they converge to a ground state with the critical mass.
\begin{proposition}
\label{prop:limE}
There results
\begin{equation}
\label{eq:Emucrit}
\ee(\overline\mu_{\overline\alpha}, \Omega) > -\infty
\end{equation}
and it is attained.
\end{proposition}
\begin{proof}
Observe first that the validity of \eqref{eq:Emucrit} is granted if one already knows that 
\begin{equation}
\label{eq:limE}
\lim_{\mu\to\overline\mu_{\overline\alpha}^-}\ee(\mu,\Omega)>-\infty\,.
\end{equation}
Indeed, suppose \eqref{eq:limE} holds but, by contradiction, $\ee(\overline\mu_{\overline\alpha}, \Omega) = -\infty$. Recalling that standard arguments (see e.g. the proof of Lemma 2.1 in \cite{DST-23}) ensure that $\ee(\mu,\Omega)$ is a strictly decreasing, concave function of $\mu$ on $[0,\overline\mu_{\overline\alpha})$, thus guaranteeing that $\displaystyle\lim_{\mu\to{\overline\mu_{\overline\alpha}}^-}\ee(\mu,\Omega)=\inf_{\mu\in(0,\overline\mu_{\overline\alpha})}\ee(\mu,\Omega)$ is well-defined, denote by $c$ such number. By density, there would then exist a nonnegative function $u \in C^\infty(\Omega)$ (the restriction to $\Omega$ of a $C^\infty_0(\R^2)$ function) such that $\|u\|_{L^2(\Omega)}^2=\overline\mu_{\overline\alpha}$ and $E(u, \Omega) \le c -2$. Let then $B \subset\subset \Omega$ be a ball where $u >0$ and let $\varphi \in C_0^\infty(\Omega)$ be positive in  $B$. By continuity, for every $\eps>0$ small,  $E(u-\eps\varphi, \Omega) \le c-1$. Moreover, if $\eps$ is small enough,
\[
0 \le u(x) -\eps\varphi(x) \le u(x) \qquad \forall x \in B
\]
so that $\|u-\eps\varphi \|_{L^2(\Omega)}^2 =: \mu < \overline\mu_{\overline\alpha}$, that in turn would yield
\[
c \le \ee(\mu,\Omega) \le E(u-\eps\varphi,\Omega) \le c-1,
\]
a contradiction.

Let us thus prove \eqref{eq:limE}. Since we already observed that the desired limit is well-defined, take $\mu_n\to{\overline\mu_{\overline\alpha}}^-$ and $u_n\in H_{\mu_n}^1(\Omega)$ be such that $E(u_n,\Omega)=\ee(\mu_n,\Omega)$ for every $n$ by Corollary \ref{attained0}, and assume by contradiction that $E(u_n,\Omega)\to-\infty$ as $n\to\infty$. Hence, $\|u_n\|_{L^4(\Omega)}\to\infty$ by the form of $E$ and $\|\nabla u_n\|_{L^2(\Omega)}\to\infty$ by \eqref{GN-bounded}. Moreover, since $u_n$ is a ground state of $E$ in $H_{\mu_n}^1(\Omega)$, with no loss of generality we can take it to be a positive solution of
\begin{equation}
	\label{NLSun}
	\begin{cases}
	-\Delta u_n+\lambda_n u_n=u_n^{3} & \text{on }\Omega\\
	\frac{\partial u_n}{\partial\nu}=0 & \text{on }\partial\Omega
	\end{cases}
\end{equation}
with
\[
\lambda_n=\frac{\|u_n\|_{L^4(\Omega)}^4-\|\nabla u_n\|_{L^2(\Omega)}^2}{\mu_n}=\frac12\frac{\|u_n\|_{L^4(\Omega)}^4}{\mu_n}-\frac2\mu\ee(\mu_n,\Omega)\to\infty\,.
\]
Moreover, $u_n$ is an action ground state of $J_4(\cdot,\lambda_n,\Omega)$ in $\NN_4(\lambda_n,\Omega)$ for every $n$ by Lemma \ref{lem-energyaction}. Hence, repeating the argument in Step 2 of the proof of Proposition \ref{asymp} above, for every $n$ we can construct a function $v_n\in H^1(\Sigma_{\overline\alpha})$ such that
\[
\begin{split}
&\,\|v_n\|_{L^2(\Sigma_{\overline\alpha})}\leq\|u_n\|_{L^2(\Omega)}^2=\mu_n<\overline\mu_{\overline\alpha}\\
&\,\|\nabla v_n\|_{L^2(\Sigma_{\overline\alpha})}\leq\|\nabla u_n\|_{L^2(\Omega)}\\
&\,\|v_n\|_{L^4(\Sigma_{\overline\alpha})}^4=\|u_n\|_{L^4(\Omega)}^4+o(1)
\end{split}
\] 
as $n\to\infty$. By Proposition \ref{prop:EGSsector}, for $n$ large enough this entails
\[
0=\ee(\mu_n,\Sigma_{\overline\alpha})\leq E(v_n,\Sigma_{\overline\alpha})\leq E(u_n,\Omega)+o(1)\,,
\]
contradicting $E(u_n,\Omega)\to-\infty$ and thus proving \eqref{eq:limE}.

Finally, to show that $\ee(\overline\mu_{\overline\alpha},\Omega)$ is attained, observe that, arguing analogously as in the first part of the proof, it coincides with $\displaystyle\lim_{\mu\to\overline\mu_{\overline\alpha}^-}\ee(\mu,\Omega)$. Hence, any sequence of energy ground states $u_n$ with mass $\mu_n\to\overline\mu_{\overline\alpha}^-$ is such that $E(u_n,\Omega)\to\ee(\overline\mu_{\overline\alpha},\Omega)$. Arguing as before shows that the corresponding frequency $\lambda_n$ is uniformly bounded from above, in turn ensuring the uniform boundedness of $u_n$ in $H^1(\Omega)$. Since $\Omega$ is bounded, the compactness of Sobolev embeddings allows to conclude that $u_n$ converges to a limit $u\in H_{\overline\mu_{\overline\alpha}}^1(\Omega)$ such that $E(u,\Omega)=\ee(\overline\mu_{\overline\alpha},\Omega)$.
\end{proof}

\begin{proof}[Proof of Proposition \ref{egs}]
	It is the content of Corollary \ref{attained0} and Proposition \ref{prop:limE}.
\end{proof}

\begin{remark}
 \label{rem-lambdapos}
 Since the previous results show that $\ee_p(\mu,\Omega)<0$ for every value of $p$ and $\mu$, it follows that the quantity $\lambda_u$ defined by \eqref{eq-lambda} is positive for every normalized ground state $u$, because
 \[
 \lambda_u=\left(1-\frac2p\right)\frac{\|u\|_{L^p(\Omega)}^p}\mu-\frac{2\ee_p(\mu,\Omega)}\mu>0
 \]
 whenever $u$ is a ground state of $E_p$ in $H_\mu^1(\Omega)$.
\end{remark}

\section{Proof of Theorem \ref{thm-main}}
\label{sec-proof}

In this section we prove Theorem \ref{thm-main}. Preliminarily, for every $p\in(2,4)$, we introduce the functions $\Lambda_p^-,\,\Lambda_p^+:(0,+\infty)\to\R$ given by
\begin{gather*}
\Lambda_p^-(\mu) := \inf\left\{\lambda_u : u  \in H^1_\mu(\Omega),\:E_p(u,\Omega)=\ee_p(\mu,\Omega) \right\}\\
\Lambda_p^+(\mu) := \sup\left\{\lambda_u : u  \in H^1_\mu(\Omega),\:E_p(u,\Omega)=\ee_p(\mu,\Omega) \right\},
\end{gather*}
with $\lambda_u$ as in \eqref{eq-lambda}. Let also
\[
W_p := \left\{\mu >0 : \Lambda_p^-(\mu) \ne \Lambda_p^+(\mu)\right\}.
\]
We collect the main properties of $\Lambda_p^\pm$ in the following lemma.

\begin{lemma}
\label{Lambda}
Let $p\in (2,4)$. Then:
\begin{enumerate}[label=(\roman*)]
\item $\Lambda_p^-(\mu),\,\Lambda_p^+(\mu)$ are attained and positive for every $\mu>0$;

\smallskip
\item if $0< \mu_1 <\mu_2$, then $\Lambda_p^+(\mu_1) < \Lambda_p^-(\mu_2)$ (in particular, $\Lambda_p^-$ and  $\Lambda_p^+$ are strictly increasing);

\smallskip
\item $W_p$ is at most countable, and the functions $\Lambda_p^-,\,\Lambda_p^+$ coincide and are continuous on $(0,+\infty) \setminus W_p$;

\smallskip
\item $\displaystyle \lim_{\mu \to +\infty} \Lambda_p^-(\mu) = +\infty\qquad\text{and}\qquad \lim_{\mu\to0^+} \Lambda_p^+(\mu) = 0$.
\end{enumerate}

\end{lemma}

\begin{proof}
To prove item \emph{(i)} it is enough to combine the arguments in the proof of Proposition \ref{agsprop}\emph{(i)} with Remark \ref{rem-lambdapos}. Moreover,
items \emph{(ii)} and \emph{(iii)} follow by \cite[Lemma 4.1 and Theorem 2.5]{DST-20} (the proofs therein are in the context of metric graphs, but they apply unchanged to the present setting).

Let us then prove item \emph{(iv)}. Since, for every fixed $\mu>0$, the constant function $\frac{\sqrt\mu}{|\Omega|}\in H_\mu^1(\Omega)$, then
\[
\ee_p(\mu,\Omega) \le E_p(\tfrac{\sqrt\mu}{|\Omega|},\Omega) = -\frac1p \frac{\mu^{\frac p2}}{|\Omega|^{\frac p2 -1}}\,.
\]
Now, let $u_\mu$ be an energy ground state with mass $\mu$ such that $\lambda_{u_\mu} = \Lambda_p^-(\mu)$. Then, by the preceding inequality and \eqref{eq-lambda},
\[
\Lambda_p^-(\mu)= \frac{\|u_\mu\|_{L^p(\Omega)}^p - \|\nabla u_\mu\|_{L^2(\Omega)}^2}{\mu}
= \left(1-\frac2p\right) \|u_\mu\|_{L^p(\Omega)}^p - 2\frac{\ee_p(\mu,\Omega)}{\mu} \ge \frac2p\left(\frac{\mu}{|\Omega|}\right)^{\frac p2-1},
\]
which proves the first limit in {\em (iv)}.

To establish the second one, let $u_\mu$ be an energy ground state with mass $\mu$ such that $\lambda_{u_\mu} = \Lambda_p^+(\mu)$. Observe that, since $\ee_p(0,\Omega) = 0$ and $\ee_p(\cdot,\Omega)$ is negative and concave on $(0,+\infty)$ (this can be seen, e.g., as in \cite[Theorem 3.1]{AST-16}), it holds $\displaystyle\lim_{\mu\to0^+}\ee_p(\mu,\Omega)=0$, so that $\ee_p(\cdot,\Omega)$ is continuous on $[0,+\infty)$. Then,
\[
0\le J_p(u_\mu,\Lambda_p^+(\mu), \Omega) = E_p(u_\mu, \Omega) + \frac\mu2 \Lambda_p^+(\mu) = \ee_p(\mu, \Omega) + \frac\mu2 \Lambda_p^+(\mu) =o(1)\quad\text{as}\quad\mu\to0^+
\]
(recall that $\Lambda_p^+(\mu)$ is bounded for small $\mu$ by \emph{(i)-(ii)}), so that
\[
o(1) = J_p(u_\mu, \Lambda_p^+(\mu),\Omega) \ge \JJ_p(\Lambda_p^+(\mu),\Omega)\geq0\quad\text{as}\quad\mu\to0^+,
\]
in turn showing that $\displaystyle\lim_{\mu\to0^+}\JJ_p(\Lambda_p^+(\mu),\Omega)=0$. The proof of item {\em(iv)} is then completed suitably using Proposition \ref{agsprop}\emph{(ii)}.
\end{proof}

The previous lemma has also the immediate consequence of allowing an alternative way to write the thesis of Theorem \ref{thm-main}.

\begin{corollary}
\label{W0}
Let $p\in (2,4)$. If $W_p\neq\emptyset$, then there exists $\mu_p>0$ and $u_1,\,u_2\in H_{\mu_p}^1(\Omega)$ such that
\[
u_1\neq u_2\qquad\text{and}\qquad E_p(u_1,\Omega)=E_p(u_2,\Omega)=\ee_p(\mu_p,\Omega)\,.
\]
In particular, $\lambda_{u_1} = \Lambda_p^-(\mu_p) \ne \Lambda_p^+(\mu_p) = \lambda_{u_2}$.
\end{corollary}

We can then state the last auxiliary result of the paper, which presents a sufficient condition for $W_p$ to be not empty.

\begin{proposition}
\label{Wne0}
Assume that, for some $\lambda_1,\,\lambda_2>0$ such that $\lambda_1 < \lambda_2$, there results $M_4^-(\lambda_1) > M_4^+(\lambda_2)$ (with $M^{\pm}_4(\cdot)$ as in Section \ref{sec-actionGS}). Then, there exists $\eps>0$ such that $W_p \ne \emptyset$ for every $p \in (4-\eps,4)$.
\end{proposition}

\begin{proof}
By Lemma \ref{semic}, the existence of $0<\lambda_1<\lambda_2$ for which $M_4^-(\lambda_1) > M_4^+(\lambda_2)$ implies that there exists $\eps>0$ such that, for every $p \in (4-\eps,4)$,
\begin{equation}
\label{M1M2}
M_p^-(\lambda_1) > M_p^+(\lambda_2).
\end{equation}
Fix then $p \in (4-\eps,4)$ and assume by contradiction that $W_p = \emptyset$. By Lemma \ref{Lambda}\emph{(ii)-(iii)}, this entails that $\Lambda_p^-(\mu) = \Lambda_p^+(\mu)=:\lambda(\mu)$ for every $\mu \in (0,+\infty)$, and that $\lambda(\cdot)$ is a continuous and strictly increasing function. Moreover, by Lemma \ref{Lambda}\emph{(iv)}, $\lambda$ is also surjective onto $(0,+\infty)$.

Now, let $u_1$ and $u_2$ be energy ground states with masses $\mu_1,\mu_2$, respectively, such that $\lambda_{u_1} = \lambda_1$ and  $\lambda_{u_2} = \lambda_2$. Since they are also action ground states at frequencies $\lambda_1,\,\lambda_2$ by Lemma \ref{lem-energyaction}, then \eqref{M1M2} yields $\mu_1 >\mu_2$, that together with the strict monotonicity of $\lambda(\cdot)$ entails $\lambda_1 > \lambda_2$, providing the contradiction we seek.
\end{proof}

We can now complete the proof of the main theorem of the paper.

\begin{proof}[Proof of Theorem \ref{thm-main}]
In view of Proposition \ref{Wne0} and Corollary \ref{W0}, it suffices to prove that there exist $0<\lambda_1<\lambda_2$ such that $M_4^-(\lambda_1) > M_4^+(\lambda_2)$.

As in the previous section, denote by $\overline\alpha$ the amplitude of the smallest internal angle of $\Omega$ and let $\overline\mu_{\overline\alpha}$ be the corresponding critical mass as in \eqref{eq:mua}.
Owing to Proposition \ref{egs}, let $\overline u\in H_{\overline\mu_{\overline\alpha}}^1(\Omega)$ be such that $E_4(\overline{u},\Omega)=\ee_4(\overline\mu_{\overline\alpha},\Omega)<0$. By Lemma \ref{lem-energyaction}, this is also an action ground state at frequency $\overline \lambda := \lambda_{\overline u}$ and $M_4^-(\overline\lambda) = M_4^+(\overline \lambda) = \overline\mu_{\overline\alpha}$. Hence, to conclude it is enough to find some $\widetilde \lambda \in(\overline \lambda,+\infty) \setminus Z_4$ (with $Z_4$ defined by Proposition \ref{agsprop}\emph{(iv)}) such that $M_4^-(\widetilde\lambda) = M_4^+(\widetilde \lambda) = : \widetilde\mu \ne \overline\mu_{\overline\alpha}$. Indeed, if $\widetilde \mu < \overline\mu_{\overline\alpha}$, we set $\lambda_1 = \overline \lambda$ and $\lambda_2 = \widetilde \lambda$, while, if $\widetilde \mu > \overline\mu_{\overline\alpha}$, we set $\lambda_1 = \widetilde \lambda$ and take $\lambda_2 \gg \lambda_1$ in such a way that $M_4^+(\lambda_2) <  \widetilde \mu$ (note that this last step is consistent since, by Proposition \ref{asymp}\emph{(ii)}, $M_4^\pm(\lambda)  \to \overline\mu_{\overline\alpha}$ as $\lambda \to +\infty$).

We are then left to show that there exists $\widetilde \lambda\in(\overline \lambda,+\infty) \setminus Z_4$ such that $M_4^-(\widetilde\lambda) = M_4^+(\widetilde \lambda) \ne \overline\mu_{\overline\alpha}$. Assume by contradiction that $M_4^-(\lambda) = M_4^+ (\lambda) =  \overline\mu_{\overline\alpha}$ for every $\lambda \in (\overline \lambda, +\infty) \setminus Z_4$. By Proposition \ref{agsprop}\emph{(iv)}, $\JJ_4(\cdot,\Omega)$ is differentiable on $(\overline \lambda, +\infty) \setminus Z_4$ and $\JJ_4'(\lambda,\Omega) = \frac{\overline\mu_{\overline\alpha}}{2}$ for every $\lambda \in (\overline \lambda, +\infty) \setminus Z_4$. Now, recalling the properties of $\overline u$, Proposition \ref{agsprop}\emph{(ii)} and Proposition \ref{prop:AGSsector}, one finds that
\begin{align*}
\JJ_4(\lambda,\Omega) = \JJ_4(\overline\lambda,\Omega) &\,+ \int_{\overline\lambda}^\lambda \JJ_4'(s,\Omega)\,ds =  \JJ_4(\overline\lambda,\Omega) + \frac{\overline\mu_{\overline\alpha}}{2}(\lambda - \overline\lambda)\\[.2cm]
&\,= E_4(\overline u, \Omega) + \JJ_4(\lambda,\Sigma_{\overline\alpha}) = \ee_4(\overline\mu_{\overline\alpha},\Omega) + \JJ_{4}(\lambda,\Sigma_{\overline\alpha}),
\end{align*}
for every $\lambda \in (\overline \lambda, +\infty) \setminus Z_4$. Hence,
\[
0 > \ee_4(\mud,\Omega) = \JJ_4(\lambda,\Omega) - \JJ_4(\lambda,\Sigma_{\overline\alpha})\qquad\forall \lambda \in (\overline \lambda, +\infty) \setminus Z_4,
\]
and, taking the limit as $\lambda \to +\infty$, we obtain a contradiction with Proposition \ref{asymp}\emph{(i)}.
\end{proof}


\section*{Statements and Declarations}
\noindent\textbf{Conflict of interest}  The authors declare that they have  no conflict of interest.

\bigskip

\noindent\textbf{Acknowledgements.} S.D. and L.T. acknowledge that this study was carried out within the project E53D23005450006 ``Nonlinear dispersive equations in presence of singularities'' - funded by European Union - Next Generation EU within the PRIN 2022 program (D.D. 104 - 02/02/2022 Ministero dell'Universit\`a e della Ricerca). This manuscript reflects only the author's views and opinions and the Ministry cannot be considered responsible for them.




\begin{thebibliography}{99}
	
\bibitem{AST-cmp}
Adami R., Serra E., Tilli P., Negative energy ground states for the $L^2$-critical NLSE on metric graphs, {\em Comm. Math. Phys.} {\bf352} (2017), 387-406.

\bibitem{AST-cvpde}
Adami R., Serra E., Tilli P., NLS ground states on graphs, {\em Calc. Var. PDE} {\bf54} (2015), 743-761.

\bibitem{AST-16}
Adami R., Serra E., Tilli P.,
Threshold phenomena and existence results for NLS ground states on metric graphs,
\emph{J. Funct. Anal.} {\bf271} (2016), no. 1, 201-223.

\bibitem{ACT}
Agostinho F., Correia S., Tavares H., Classification and stability of positive solutions to the NLS equation on the $\mathcal T$-metric graph, {\em Nonlinearity} {\bf37} (2024), no. 2, 025005.

\bibitem{bandle}
Bandle C., Isoperimetric inequalities and applications, Pitman, London, 1980.

\bibitem{BdV13}
Bartsch T., de Valeriola S., Normalized solutions of nonlinear Schr\"odinger equations, {\em Arch. Math.}  {\bf100}
(2013), no. 1,  75-83.

\bibitem{BJ}
Bartsch T., Jeanjean L., Normalized solutions for nonlinear Schr\"odinger systems, {\em Proc. Royal Soc. Edinb. Section A: Math.} {\bf148}
(2018), no. 2, 225-242.

\bibitem{BJS}
Bartsch T., Jeanjean L., Soave N., Normalized solutions for a system of coupled cubic Schr\"odinger equations on
$\R^3$, {\em J. Math. Pures Appl.} {\bf106} (2016), no. 4, 583-614.

\bibitem{BS17}
Bartsch T., Soave N., A natural constraint approach to normalized solutions of nonlinear Schr\"odinger equations and
systems, {\em J. Funct. Anal.} {\bf272} (2017), no. 12, 4998-5037.

\bibitem{BS19}
Bartsch T., Soave N., Multiple normalized solutions for a competing system of Schr\"odinger equations, {\em Calc. Var. PDE} {\bf58} (2019), no. 1,  art. 22, 24.

\bibitem{BL-83}
Berestycki H., Lions P.-L.,
Existence of solutions for nonlinear scalar field equations, I existence of a ground state,
\emph{Arch. Rational Mech. Anal.} {\bf 82} (1983), no. 4, 313-345.

\bibitem{BCJS}
Borthwick J., Chang X., Jeanjean L., Soave N., Normalized solutions of $L^2$-supercritical NLS equations on noncompact metric graphs with localized nonlinearities, {\em Nonlinearity} {\bf36} (2023), 3776-3795.

\bibitem{BL83}
Brezis H., Lieb E.H., A relation between pointwise convergence of functions and convergence of functionals,
{\em Proc. Amer. Math. Soc.} {\bf 88}(3)  (1983), 486-490.

\bibitem{BZ}
Brothers J.E., Ziemer W.P., Minimal rearrangements of Sobolev functions, {\em J. Reine Angew. Math.} {\bf384} (1988), 153-179.

\bibitem{CL82}
Cazenave T., Lions P.L., Orbital stability of standing waves for some nonlinear Schr\"odinger equations, {\em Comm. Math. Phys.} {\bf85} (1982), 549-561.

\bibitem{cazenave}
Cazenave T., Semilinear Schr\"odinger Equations, Courant Lecture Notes, vol.10, American Mathematical Society, Providence, RI, 2003.

\bibitem{CJS}
Chang X., Jeanjean L., Soave N., Normalized solutions of $L^2$-supercritical NLS equations on compact metric graphs, {\em Ann. Inst. H. Poincar\'e C Anal. Non Lin\'eaire} {\bf41} (2023), no. 4,  933-959.

\bibitem{CZR}
Coti Zetati V., Rabinowitz P., Homoclinic Type Solutions
for a Semilinear Elliptic PDE on $\R^n$, {\em Comm. Pure Appl. Math.} {\bf XLV} (1992), 1217-1269.

\bibitem{DDGS}
De Coster C., Dovetta S., Galant D., Serra E., An action approach to nodal and least energy normalized solutions for nonlinear Schr\"odinger equations,  {\em Ann. Inst. H. Poincar\'e C Anal. Non Lin\'eaire}, 10.4171/AIHPC/160.

\bibitem{D-24}
Dovetta S.,
Non-uniqueness of normalized ground states for nonlinear Schr\"odinger equations on metric graphs, {\em Proc. London Math. Soc.} {\bf130} (2025), no. 2, e70025.

\bibitem{DST-20}
Dovetta S., Serra E., Tilli P.,
Uniqueness and non-uniqueness of prescribed mass NLS ground states on metric graphs,
\emph{Adv. Math.} {\bf 374} (2020), 107352, 41pp.

\bibitem{DST-23}
Dovetta S., Serra E., Tilli P.,
Action versus energy ground states in nonlinear Schr\"odinger equations,
\emph{Math. Ann.} {\bf 385} (2023), no. 3-4, 1545-1576.

\bibitem{GJ}
Gou T., Jeanjean L., Multiple positive normalized solutions for nonlinear Schr\"odinger systems, {\em Nonlinearity} {\bf31}
(2018), no. 5, 2319.

\bibitem{J97}
Jeanjean L., Existence of solutions with prescribed norm for semilinear elliptic equations, {\em Nonlin. Anal. } {\bf28} (1997), no. 10,  1633-1659.

\bibitem{JL-22}
Jeanjean L., Lu S.-S.,
On global minimizers for a mass constrained problem,
\emph{Calc. Var. PDE} {\bf 61} (2022), no. 6, art. n. 214, 18 pp.

\bibitem{K18}
Kosaka A., Condensation phenomena of a least-energy solution to semilinear Neumann problems on piecewise smooth domains, {\em J. Diff. Eq.} {\bf265} (2018), 1-32.

\bibitem{K-89}
Kwong M.K.,
Uniqueness of positive solutions of $\Delta u-u+u^p=0$ in $\R^n$,
\emph{Arch. Rational Mech. Anal.} {\bf 105} (1989), no. 3, 243-266.

\bibitem{PL2}
Lions P.-L.,
The concentration-compactness principle in the calculus
of variations. The locally compact case, part 2,
 {\em Ann. Inst. H. Poincar\'e C Anal. Non Lin\'eaire} {\bf 1} (1984),  223-283.

\bibitem{LLZ}
Li X., Liu L., Zhang G., Ground states for the NLS equation with combined local nonlinearities on noncompact metric graphs, {\em J. Math. Anal. Appl.} {\bf530} (2024), no. 1, 127672.

\bibitem{NT2}
Ni W.-M., Takagi I., Locating the peaks of least-energy solutions to a semilinear Neumann problem, {\em Duke Math. J}. {\bf70} (1993), 247-281.

\bibitem{NT1}
Ni W.-M., Takagi I., On the shape of least-energy solutions to a semilinear Neumann problem, {\em Comm. Pure Appl. Math.} {\bf44} (1991), no. 7, 819-851.

\bibitem{NP}
Noja D., Pelinovsky D.E., Standing waves of the quintic NLS equation on the tadpole graph, {\em Calc. Var. PDE} {\bf59} (2020), no. 5,  art. n. 173.

\bibitem{NTV-14}
Noris B., Tavares H., Verzini G.,
Existence and orbital stability of the ground states with prescribed mass for the $L^2$-critical and supercritical NLS on bounded domains,
\emph{Anal. PDE} {\bf 7} (2014), no. 8, 1807-1838.

\bibitem{NTV15}
Noris B., Tavares H., Verzini G., Stable solitary waves with prescribed $L^2$-mass for the cubic Schr\"odinger system with trapping potentials, {\em Discrete Contin. Dyn. Syst.} {\bf35} (2015), no. 12, 6085-6112.

\bibitem{NTV19}
Noris B., Tavares H., Verzini G., Normalized solutions for nonlinear Schr\"odinger systems on bounded domains,
{\em Nonlinearity} {\bf32} (2019), no. 3, 1044-1072.

\bibitem{PT}
Pacella F., Tralli G., Isoperimetric cones and minimal solutions of partial overdetermined problems, {\em Publ. Mat.} {\bf65} (2021), no. 1, 61-81.  

\bibitem{PPVV}
Pellacci B., Pistoia A., Vaira G., Verzini G., Normalized concentrating solutions to nonlinear elliptic problems, {\em J. Diff. Eq.} {\bf275} (2021), 882-919.

\bibitem{PS}
Pierotti D., Soave N., Ground states for the NLS equation with combined nonlinearities on non-compact metric graphs, {\em SIAM J. Math. Anal.} {\bf54} (2022), no. 1,  768-790.

\bibitem{PSV}
Pierotti D., Soave N., Verzini G., Local minimizers in absence of ground states for the critical NLS energy on metric graphs, {\em Proc. Royal Soc. Edinb. Section A: Math.} {\bf151} (2021), no. 2,  705-733.

\bibitem{PV}
Pierotti D., Verzini G., Normalized bound states for the nonlinear Schr\"odinger equation in bounded domains, {\em Calc.
Var. PDE} {\bf56} (2017), no. 5,  art. n. 133, 27.

\bibitem{PVY}
Pierotti D., Verzini G., Yu J., Normalized solutions for Sobolev critical Schr\"odinger equations on bounded domains, {\em SIAM J. Math. Anal.} {\bf 57} (2025), no. 1, 262-285.

\bibitem{ST}
Serra E., Tentarelli L., Bound states of the NLS equation on metric graphs with localized nonlinearities, {\em J. Diff. Eq.} {\bf260} (2016), no. 7, 5627-5644.

\bibitem{S1}
 Soave N., Normalized ground states for the NLS equation with combined nonlinearities, {\em J. Diff. Eq.} {\bf269}
(2020), no. 9,  6941-6987.

\bibitem{S2}
Soave N., Normalized ground states for the NLS equation with combined nonlinearities: the Sobolev critical case,
{\em J. Funct. Anal.} {\bf279} (2020), no. 6, 108610.

\bibitem{SZ1}
Song L., Zou W., On multiple sign-changing normalized solutions to the Brezis-Nirenberg problem, {\em Math. Zeitschrift} {\bf309} (2025), art. 61.

\bibitem{SZ2}
Song L., Zou W., Two Positive Normalized Solutions on Star-shaped Bounded Domains to the Brezis-Nirenberg Problem, I: Existence, (2024) arXiv:2404.11204.

\bibitem{T}
Tentarelli L., NLS ground states on metric graphs with localized nonlinearities, {\em J. Math. Anal. Appl.} {\bf433} (2016), no. 1, 291-304.

\end{thebibliography}
\end{document}